\pdfoutput=1

\documentclass[reqno,12pt]{amsproc}

\usepackage{amsbsy}
\usepackage{amsfonts}
\usepackage{amsmath}
\usepackage{amsthm}
\usepackage{amssymb}
\usepackage{enumerate}

\usepackage{hyperref}
\usepackage{ytableau}
\usepackage{color}
\usepackage{etoolbox}
\expandafter\patchcmd\csname\string\proof\endcsname
{\normalparindent}{0pt }{}{}
\usepackage{graphicx}

\usepackage[margin=1.5in]{geometry}
	
	\newtheorem{thm}{Theorem}[section]
	\newtheorem{lem}[thm]{Lemma}
	\newtheorem{prop}[thm]{Proposition}
	\newtheorem{cor}[thm]{Corollary}
	\newtheorem{conj}[thm]{Conjecture}
	\newtheorem{dfn}[thm]{Definition}
	
	\newcommand\MOD{\textrm{ (mod }}
	\newcommand\Var{\mathrm{Var}}
	\newcommand\Covar{\mathrm{Covar}}

	\newcommand\Tr{\mathrm{Tr}}

	\newcommand{\vgeq}{\rotatebox[origin=c]{-90}{$\geq$}}

\begin{document}

\title{Arithmetic functions in short intervals and the symmetric group}
\author{Brad Rodgers}
\address{Department of Mathematics, University of Michigan \\  530 Church St., Ann Arbor, MI 48109}
\email{rbrad@umich.edu}
\date{}

\begin{abstract}
	We consider the variance of sums of arithmetic functions over random short intervals in the function field setting. Based on the analogy between factorizations of random elements of $\mathbb{F}_q[T]$ into primes and the factorizations of random permutations into cycles, we give a simple but general formula for these variances in the large $q$ limit for arithmetic functions that depend only upon factorization structure. From this we derive new estimates, quickly recover some that are already known, and make new conjectures in the setting of the integers.

	In particular we make the combinatorial observation that any function of this sort can be decomposed into a sum of functions $u$ and $v$, depending on the size of the short interval, with $u$ making a negligible contribution to the variance, and $v$ asymptotically contributing diagonal terms only. 
	
	This variance evaluation is closely related to the appearance of random matrix statistics in the zeros of families of L-functions and sheds light on the arithmetic meaning of this phenomenon.
\end{abstract}


\maketitle

\section{Historical Background and Motivation}
\label{sec:1}

The purpose of this paper is to explore a connection between two well-known phenomena in number theory: that the zeros of a family of L-functions distribute like the eigenvalues of a random matrix and that the prime factors of a random integer distribute like the cycles of a random permutation. We use this connection to give a general yet simple description for the statistical behavior of sums of arithmetic functions over short intervals. The results that we ultimately prove will make use of a function field analogy: they concern arithmetic functions defined on $\mathbb{F}_q[T]$ rather than the integers and we will require that $q\rightarrow\infty$. We begin in this section however with a discussion of some historical conjectures and heuristics from the integers that motivate what follows. A statement of the most important results we prove may be found at the beginning of section \ref{sec:3} -- our main results are Theorems \ref{thm:general_var} and \ref{thm:subspace_decompose1} along with Corollary \ref{cor:variance_to_minimizer}. Key use is made of a combinatorial variant of the explicit formula of Weil, Theorem \ref{thm:schur_in_zeros}, which may be of independent interest.

We recall the following conjectures:

\begin{conj} [Good-Churchhouse \cite{GoCh}]
	\label{conj:GoCh}
	As $X\rightarrow\infty$, for $H = X^\delta$ with $\delta \in (0,1)$,
	$$
	\frac{1}{X} \int_X^{2X} \bigg( \sum_{x \leq n \leq n+H} \mu(n)\bigg)^2 \,dx \sim \frac{6}{\pi^2} H.
	$$
\end{conj}

\begin{conj}[Goldston-Montgomery \cite{GoMo}]
	\label{conj:GoMo}
	As $X\rightarrow\infty$ for $H = X^\delta$ with $\delta \in (0,1)$,
	$$
	\frac{1}{X}\int_X^{2X} \bigg( \sum_{x \leq n \leq x +H} \Lambda(n) - H\bigg)^2\, dx \sim H (\log X  -\log H).
	$$
\end{conj}

In both conjectures, we consider random $x \in [X,2X]$ and seek to compute the variance of the sum of an arithmetic function, $\mu(n)$ or $\Lambda(n)$, over the random short interval $[x,x+H]$. Here $\mu(n)$ is the M\"obius functions, which oscillates around the value $0$, and $\Lambda(n)$ is the von Mangoldt function which has an average value of $1$, by the prime number theorem. Similar conjectures can be made for, for instance, the higher order von Mangoldt functions $\Lambda_j(n)$ \cite{Ro} or the $k$-fold divisor function $d_k(n)$ \cite{KeRoRoRu}, the latter of which is conjectured to display a very curious series of `phase changes' as the parameter $\delta$ varies. These conjectures are known to be closely related to the conjectural phenomenon that the zeros of families of $L$-functions tends to distribute like the eigenvalues of certain random matrices (see \cite{KaSa} for an exposition on the latter phenomenon).

In the past few years, beginning with the work of Keating and Rudnick \cite{KeRu1}, function field variants of these conjectures have been proved. (In some cases the function field theorems have in fact motivated new conjectures.) In order to state these function field results, we make use of a well-known dictionary between the integers $\mathbb{Z}$ and the ring of polynomials over a finite field, that is $\mathbb{F}_q[T]$. To review this dictionary and fix some of our notation:
\begin{itemize}
	\item The collection of monic polynomials, $\mathcal{M}$, takes the place of positive integers,
	\item The degree, $\deg(f)$, of $f \in \mathcal{M}$ takes the place of $\log n$ for $n \in \mathbb{N}$.
	\item The collection of degree $n$ monic polynomials, $\mathcal{M}_n$, takes the place of integers lying in a dyadic interval $[X,2X]$.
	\item Irreducible polynomials take the role of primes.
	\item For $f \in \mathcal{M}$ and $h < \deg(f)$, the set $I(f;h):=\{g \in \mathcal{M}\,:\, \deg(f-g) \leq h\}$ is a short interval around the polynomial $f$, playing the role of $[x,x+H]$. (Here $h$ may be thought of as corresponding to $\log H$.)
\end{itemize}

\noindent Note that $|\mathcal{M}_n| = q^n$, while $|I(f;h)| = q^{h+1}.$

This set-up is explained more extensively in, for instance, the ICM address of Rudnick \cite{Ru} or the book of Rosen \cite{Ro}. We have the following analogues of Conjecture \ref{conj:GoCh} and \ref{conj:GoMo}:

\begin{thm}[Keating-Rudnick \cite{Ru}, Bae-Cha-Jung \cite{BaChJu}]
	\label{thm:a_mobiusvar}
	For fixed $0 \leq h \leq n-5$, as $q\rightarrow\infty$,
	\begin{equation}
	\label{a_mobiusvar}
	\frac{1}{q^n} \sum_{f \in \mathcal{M}_n} \bigg| \sum_{g \in I(f;h)} \mu(g) \bigg|^2 \sim q^{h+1}.
	\end{equation}
\end{thm}

\begin{thm}[Keating-Rudnick \cite{KeRu1}]
	\label{thm:b_mangoldtvar}
	For fixed $0 \leq h \leq n-5$ as $q\rightarrow\infty$,
	\begin{equation}
	\label{b_mangoldtvar}
	\frac{1}{q^n} \sum_{f \in \mathcal{M}_n} \bigg| \sum_{g \in I(f;h)} \Lambda(g) - q^{h+1}\bigg|^2 \sim q^{h+1}(n-h-2).
	\end{equation}
\end{thm}

For $g\in\mathcal{M}$, the M\"obius function $\mu(g)$ is defined in analogy with the integers by $\mu(g) = (-1)^\ell$ if $g$ is squarefree (that is $g$ has no repeated factors) and $g = P_1\cdots P_\ell$ in its prime factorization, and $\mu(g) = 0$ if $g$ is squareful\footnote{There is a closely related terminology `square-full', which means something quite different -- namely that for prime $P$, if $P|g$, we have $P^2|g$ also. The distinction is important to keep in mind. Square-full numbers will not play a role in this paper.} (that is $g$ is not squarefree). Likewise $\Lambda(g) = \deg(P)$ if $g = P^k$ for a prime $P$ and a power $k \geq 1$, and $\Lambda(g) = 0$ otherwise.	

We introduce a notation to write these results more succinctly. For a function $\eta: \mathcal{M}_n \rightarrow \mathbb{C}$, we define its mean value by
\begin{equation}
	\label{def:expect}
	\mathbb{E}_{f\ \in \mathcal{M}_n} \eta(f):= \frac{1}{q^n} \sum_{f \in \mathcal{M}_n} \eta(f),
\end{equation}
and its variance by
\begin{equation}
	\label{def:variance}
	\Var_{f\in\mathcal{M}_n}\Big(\eta(f)\Big):= \frac{1}{q^n} \sum_{f\in \mathcal{M}_n} \big| \eta(f) - \mathbb{E}_{\mathcal{M}_n}\eta \,\big|^2.
\end{equation}
Note that both the mean value and variance typically depend on the size of the field $q$. As a test of notation, the reader may easily verify that
\begin{equation}
	\label{c_1var}
	\Var_{f\in\mathcal{M}_n}\bigg(\sum_{g\in I(f;h)} 1 \bigg) = 0.
\end{equation}
Likewise we see that \eqref{a_mobiusvar} may be rewritten
\begin{equation}
	\label{d_mobiusvar}
	\Var_{f\in\mathcal{M}_n}\bigg(\sum_{g\in I(f;h)} \mu(g) \bigg) \sim q^{h+1},
\end{equation}
and \eqref{b_mangoldtvar},
\begin{equation}
	\label{e_mangoldtvar}
	\Var_{f\in\mathcal{M}_n}\bigg(\sum_{g\in I(f;h)} \Lambda(g) \bigg) \sim q^{h+1}(n-h-2),
\end{equation}
as $q\rightarrow\infty$. 

We may add another recent result \cite{KeRoRoRu} to this list, due to Keating, the author, Roditty-Gershon, and Rudnick, for the $k$-fold divisor function, which is defined in analogy with the integers by $d_k(f) := |\{(a_1,...,a_k) \in \mathcal{M}^k:\, f = a_1\cdots a_k\}|$.

\begin{thm}
	\label{thm:f_dvar}
	For fixed positive integer $k$, and fixed $0 \leq h \leq n-5$, as $q\rightarrow\infty$,
	\begin{equation}
	\label{f_dvar}
	\Var_{f\in\mathcal{M}_n}\bigg(\sum_{g\in I(f;h)} d_k(g)\bigg) = q^{h+1} \mathcal{I}_k(n,\, n-h-2) + O(q^{h+1/2}),
	\end{equation}
	where $\mathcal{I}_k(m, \,N)$ is the count of lattice points $(x_{ij}) \in (\mathbb{Z})^{k^2}$ satisfying each of the following conditions,
	\begin{enumerate}
		\item $0 \leq x_{ij} \leq N$ for all $1 \leq i,j \leq k$,
		\item $x_{11} + \cdots + x_{kk} = m$,
		\item The array $x_{ij}$ is weakly decreasing across columns and down rows. That is
		$$
		\begin{matrix}
		x_{11}& \geq & x_{12}& \geq & \cdots & \geq & x_{1k}\\
		\vgeq & & \vgeq & & & &\vgeq \\
		x_{21}& \geq &  x_{22}& \geq &\cdots & \geq & x_{2k} \\
		\vgeq & & \vgeq & & & &\vgeq \\
		\vdots& & \vdots & & \ddots & & \vdots \\
		\vgeq & & \vgeq & & & &\vgeq \\
		x_{k1}& \geq & x_{k2}& \geq & \cdots & \geq & x_{kk}\\
		\end{matrix}
		$$
	\end{enumerate}
\end{thm}

Of the evaluations \eqref{c_1var} through \eqref{f_dvar}, only \eqref{c_1var} may proved easily (in fact trivially). Nonetheless, the estimate in \eqref{d_mobiusvar}, while deep, at least has a heuristic meaning that is easy to understand; it is just the claim that in expanding the variance into a sum over two indices, the M\"obius function is so oscillatory that off-diagonal terms make no contribution. That is, \eqref{d_mobiusvar} may be understood heuristically in the following way:
\begin{align*}
\Var_{f\in\mathcal{M}_n}\bigg( \sum_{g\in I(f;h)} \mu(g)\bigg) &= \frac{q^{h+1}}{q^n} \sum_{\substack{g_1, g_2 \\ \deg(g_1-g_2) \leq h}} \mu(g_1)\mu(g_2) \\
&\approx \frac{q^{h+1}}{q^n} \sum_{\substack{g_1,g_2 \in \mathcal{M}_n \\ g_1 = g_2}} \mu(g_1)\mu(g_2).
\end{align*}
See for instance \cite{Ng} for a broader application of this heuristic in the setting of the integers.

The evaluation of the $k$-fold divisor function in \eqref{f_dvar} is obviously of a more complicated sort, even heuristically. In particular it may be seen that $\mathcal{I}_k(n;\, n-h-2)$ is a piecewise polynomial, and for $k\geq 3$ as $h$ ranges from $0$ to $n-5$, it exhibits several phase changes in its behavior in various ranges of $h$ (see \cite[Sec. 4]{KeRoRoRu}). The arithmetic reason for these phase changes in particular is rather mysterious.

Nonetheless, we make the following claim: \eqref{f_dvar} may be understood arithmetically as nothing more complicated than a combination of the phenomena that give rise to \eqref{c_1var} and \eqref{d_mobiusvar}. For any degree $n$ and short interval size $h$, we will observe that we may decompose
$$
d_k(f) = u(f) + v(f),
$$
where $u$ and $v$ are arithmetic functions, with $u(f)$ regular enough within the specified short intervals that (in analogy with \eqref{c_1var}),
\begin{equation}
\label{g_uvar}
\Var_{f\in\mathcal{M}_n}\bigg( \sum_{g\in I(f;h)} u(g)\bigg) = o(q^{h+1}),
\end{equation}
while $v(f)$ is oscillatory and
\begin{equation}
\label{h_vvar}
\Var_{f\in\mathcal{M}_n} \bigg( \sum_{g\in I(f;h)} v(g) \bigg) \sim q^{h+1} \cdot \frac{1}{q^{n+1}}\sum_{g\in \mathcal{M}_n} |v(g)|^2.
\end{equation}
That is, as with the M\"obius funcion, only diagonal terms contribute to its variance, in analogy with \eqref{d_mobiusvar}.

From Cauchy-Schwarz, it follows that
$$
\Var_{f\in\mathcal{M}_n} \bigg( \sum_{g\in I(f;h)} d_k(g) \bigg) \sim q^{h+1} \cdot \frac{1}{q^n}\sum_{g\in \mathcal{M}_n}\; |v(g)|^2.
$$

This decomposition is explicit, based on symmetric function theory, and is given below -- the quantity $\mathcal{I}_k(n;n-h-2)$ may be recovered from it. That \eqref{g_uvar} holds for our function $u$ will be a relatively shallow fact (having to do with the number of zeros of a certain family of $L$-functions), and one may think of $u$ as being the largest piece of $d_k$ with enough regularity that \eqref{g_uvar} holds for this reason. The intricacies of the variance estimate in \eqref{f_dvar} may thus be thought of as resulting from the fact that this decomposition changes for various values of $n$ and $h$.

Such a decomposition is not limited to the $k$-fold divisor function. Any arithmetic functions whose value depends only upon the factorization type of its argument may be decomposed in this way and the variance of its sum over short intervals may thus be evaluated. What me mean by factorization type is defined formally below; roughly this is the size of all prime factors, listed with multiplicity. The functions $\mu(f)$, $\Lambda(f)$, $\Lambda_j(f)$, and $d_k(f)$ are all examples to which the result may be applied.

The evaluation of variance for such a general class of function is closely related to the known phenomena that the zeros of $L$-functions distribute like the eigenvalues of a random matrix. Indeed, the result we prove may be seen to be an equivalent restatement of an equidistribution result of Katz, Theorem \ref{katz} below. (We make use of Katz's Theorem \ref{katz} in our proof, and so we \emph{do not} arrive at an independent proof of it however.) 

We will use this general variance evaluation to recover several of the results that have been mentioned above with relatively little extra work and to derive new results that seem difficult by other means. New conjectures in the setting of the integers are put forward based on these results. Perhaps of particular interest, we consider sums of the function $\omega(n)$, counting prime factors: based on a function field model, we conjecture that the variance of sums of this function is somewhat smaller than a naive heuristic would lead one to believe.

In addition to yielding a pleasant general formula, the decomposition results of this paper help elucidate why random matrix universality should make an appearance in number theory. A complementary perspective as to the arithmetic reasons for the appearance of random matrix theory in number theory, dealing with the integers themselves, has appeared in the work of Bogomolny and Keating \cite{BoKe1,BoKe2} and in work of Conrey and Keating \cite{CoKe1,CoKe2,CoKe3,CoKe4}. It would be very interesting to see if the combinatorial decompositions in the present paper can be extended to the setting of the integers in a way consistent with various conjectures that have been made there.


We finally note a recent application of our main results to algebraic geometry proper; by combining Theorem \ref{thm:general_var} with other work of their own, Hast and Matei \cite{HaMa} have given a geometric interpretation of this result. Indeed, it may be possible to prove Theorem \ref{thm:general_var} of this paper rather more directly through algebro-geometric means; we hope nonetheless that the proof given here will be of use especially to analytic number theorists more comfortable with the ideas we will make use of.

\section{The symmetric group and factorization type}
\label{sec:2}

\subsection{}
\label{sec:intro2}
The decomposition described in section \ref{sec:1} and the corresponding estimates for variance hinge upon a well known analogy between the prime factors of a random integer or element of $\mathbb{F}_q[T]$ and the cycles of a random permutation. (Later an application of symmetric function theory to the zeros of $L$-functions will play an equally important and dual role.)

We begin by recalling how it is that factorizations over $\mathbb{F}_q[T]$ resemble the cycles of permutations.

Recall that $\mathcal{M}_n$, the collection of monic polynomials of degree $n$, consists of $q^n$ elements. Recall also that a partition $\lambda$ of a positive integer $n$ is defined to be a sequence of non-increasing positive integers $(\lambda_1, \lambda_2,...,\lambda_k)$ such that for $|\lambda|:= \lambda_1+\cdots+\lambda_k$ we have $|\lambda| = n$. We will also use the notation $\lambda \vdash n$ to indicate that $\lambda$ is a partition of $n$.

\begin{dfn}
	\label{dfn:factorizationtype}
	For an element $f$ of $\mathcal{M}_n$ that is squarefree, if $f$ has prime factorization $f = P_1 P_2 \cdots P_k$ with $\deg P_1 \geq \deg P_2 \geq ...$, we define the \emph{factorization type} to be the partition of $n$ given by
	$$
	\tau_f = (\deg P_1,...,\deg P_k).
	$$
	For $f$ that is not squarefree (i.e. squareful) we adopt the convention that $\tau_f = \emptyset$ (the empty partition).
\end{dfn}
In the above definition we have fixed our attention on the squarefrees because as $q\rightarrow\infty$ nearly all elements of $\mathcal{M}_n$ are squarefree (see \cite[Prop 2.3]{Ros}, or \cite[Thm. 4.1]{Weiss}):
\begin{equation}
	\label{eq:squarefree_density}
	\frac{1}{q^n} \#\{f \in \mathcal{M}_n:\, f \;\textrm{squarefree}\} = 1 - O(1/q).
\end{equation}

Note that likewise any element $\sigma$ of the the symmetric group $\mathfrak{S}_n$ on $n$ elements can be written uniquely as a product of disjoint cycles: $\sigma = \sigma_1\sigma_2\cdots \sigma_k$. Denote the lengths of the cycles by $|\sigma_i|$. For instance $|(245)| = 3$, where we have used cycle notation to represent the permutation.
\begin{dfn}
	\label{dfn:cycletype}
	For an element $\sigma\in \mathfrak{S}_n$, with $\sigma = \sigma_1\sigma_2\cdots \sigma_k$ and $|\sigma_1| \geq |\sigma_2| \geq ...$ we define the \emph{cycle type} to be the partition of $n$ given by
	$$
	\tau_\sigma = (|\sigma_1|,...,|\sigma_k|).
	$$	
\end{dfn}

It is well-known that as $q\rightarrow\infty$ the distribution over $\mathcal{M}_n$ of factorization types tends to the distribution of cycle types in $\mathfrak{S}_n$ \cite{AnBaRu}:
\begin{prop}
	\label{prop:factor_to_cycle}
	For a partition $\lambda \vdash n$,
	$$
	\lim_{q\rightarrow\infty}\mathbb{P}_{f \in \mathcal{M}_n}(\tau_f = \lambda) = \mathbb{P}_{\sigma \in \mathfrak{S}_n}(\tau_\sigma = \lambda).
	$$
\end{prop}

Here and in what follows we have used elementary probabilistic notation, for instance:
$$
\mathbb{P}_{f \in \mathcal{M}_n}(\tau_f = \lambda) := \frac{1}{q^n} \#\{f\in \mathcal{M}_n:\, \tau_f = \lambda\}.
$$ 

There is a well-known expression for the probability that a random permutation has a cycle structure $\lambda$, due to Cauchy. We use the standard partition frequency notation $\lambda = \langle 1^{m_1} 2^{m_2} \cdots j^{m_j} \rangle$; this means for $\lambda = (\lambda_1,\lambda_2,...)$, that $m_1$ of the parts of $\lambda$ are equal to $1$, $m_2$ are equal to $2$, etc. So if $\tau_\sigma = \langle 1^{m_1}2^{m_2}\cdots j^{m_j}\rangle$, $\sigma$ has $m_1$ $1$-cycles, $m_2$ $2$-cycles, etc. With this notation, Cauchy's result is that
\begin{equation}
	\label{eq:cycleprob}
	\mathbb{P}_{\sigma \in \mathfrak{S}_n}(\tau_\sigma = \lambda) = \mathbf{p}(\lambda), \quad \textrm{where}\quad  \mathbf{p}(\lambda) := \prod_{i=1}^j \frac{1}{i^{m_i}m_i!}
\end{equation}

It is worth mentioning a recent result of Andrade, Bary-Soroker, and Rudnick \cite{AnBaRu} that has generalized this picture. They show that the factorization types of a random polynomial $f$ and a shift $f + \alpha$  become independent as $q\rightarrow\infty$:
\begin{thm}[Andrade -- Bary-Soroker -- Rudnick]
	\label{thm:factorizationindependence}
	For partitions $\lambda, \nu \vdash n$, uniformly for $\deg(\alpha) < n$,
	$$
	\mathbb{P}_{f \in \mathcal{M}_n}(\tau_f = \lambda,\,\tau_{f+\alpha} = \nu) = \mathbf{p}(\lambda)\mathbf{p}(\nu) + O(q^{-1/2}).
	$$
\end{thm}
In fact they demonstrate this independence even for multiple shifts: the factorization types of $f+\alpha_1, f+ \alpha_2, ..., f+\alpha_k$ become independent as well.

\subsection{}
\label{sec:intro3}
In this paper we will be concerned with the distribution of arithmetic functions $a: \mathcal{M} \mapsto \mathbb{C}$ such that $a(f)$ depends only upon the size and exponents of the prime factors of $f$. To make a more formal definition, if $f$ has prime factorization $ P_1^{e_1}\cdots P_k^{e_k}$, with $P_1,...,P_k$ monic primes, we call the data $(\deg P_1, e_1; \cdots ;\allowbreak \deg P_k, e_k)$, the \emph{extended factorization type} of $f$. We will be concerned with functions $a$ such that $a(f)$ depends only on the extended factorization type of $f$, and we call such functions \emph{factorization functions}. The class of factorization functions includes, for instance, the M\"obius function $\mu(f)$, the von Mangoldt function $\Lambda(f)$, the count-of-divisors function $d(f)$, the indicator function of degree $n$ polynomials $\mathbf{1}[\deg(f) = n]$, the indicator function of squarefree polynomials $\mu(n)^2$, etc. It does not include Dirichlet characters, for instance.

It is evident that for each $n$, the linear space of factorization functions supported on degree $n$ polynomials is of finite dimension. The space of factorization functions supported on degree $n$ squarefree polynomials is likewise of (smaller) finite dimension. In invoking the symmetric group, Proposition \ref{prop:factor_to_cycle} and Theorem \ref{thm:factorizationindependence} suggest that the space of factorization functions has an important basis that may provide useful information: namely the irreducible characters of $\mathfrak{S}_n$. 

In describing how such characters may be applied to elements of $\mathbb{F}_q[T]$, we suppose the reader is familiar with the most basic outlines of representation theory over $\mathfrak{S}_n$, along the lines of for instance Chapter 4 of \cite{FuHa}. We recall that the space of class functions of $\mathfrak{S}_n$ are those functions $a(\sigma)$ with a value depending only on the cycle type the permutation $\sigma$ and that a basis for such functions is given by the irreducible characters, for which we use the notation\footnote{We use the letter $X$ rather than the more traditional $\chi$ to distinguish these characters from Dirichlet characters which will make an appearance later on.} 
$$
X^\lambda(\sigma).
$$
If $\sigma$ has cycle type $\tau$, sometimes instead of $X^\lambda(\sigma)$ we write $X^\lambda(\tau)$, since $X^\lambda$ depends only on cycle type. Such characters are indexed by partitions $\lambda \vdash n$, and there is a one-to-one correspondence between irreducible characters of $\mathfrak{S}_n$ and partitions of $n$. These characters satisfy the orthogonality relation:
\begin{equation}
	\label{eq:ortho_sym}
	\mathbb{E}_{\sigma\in \mathfrak{S}_n} X^{\lambda_1}(\sigma) X^{\lambda_2}(\sigma) = \delta_{\lambda_1 = \lambda_2}.
\end{equation}

For an element $f\in \mathbb{F}_q[T]$, for $\lambda \vdash n$, we define
$$
X^\lambda(f) := \begin{cases} X^\lambda(\tau_f) & \text{for}\, \text{if} \deg(f) = n\, \text{and}\, f\, \text{is squarefree} \\
0 & \text{otherwise}.
\end{cases}
$$

It is easy to see from this definition that for any factorization function $a$, there exists a unique decomposition
\begin{equation}
	\label{eq:factfourier}
	a(f) = \sum_{\lambda} \hat{a}_\lambda X^{\lambda}(f) + b(f),
\end{equation}
where $b(f)$ is a function supported on the squarefuls, $\hat{a}_\lambda$ are constants that depend on the function $a$ and are defined by this relation, and the sum is over all partitions. (Note that for any particular $f$ of degree $n$, the sum in \eqref{eq:factfourier} will be a finite sum over $\lambda \vdash n$, all other terms in the summand being $0$.)

Note that from Proposition \ref{prop:factor_to_cycle} and the orthogonality relation \eqref{eq:ortho_sym} we may equivalently define the coefficients $\hat{a}_\lambda$ for $\lambda \vdash n$ by
$$
\hat{a}_\lambda := \lim_{q\rightarrow\infty} \frac{1}{q^n}\sum_{f \in \mathcal{M}_n} a(f) X^{\lambda}(f).
$$
For instance, since $X^{(n)}$ is the trivial character, we have $\mathbb{E}_{f\in \mathcal{M}_n} a(f) \rightarrow \hat{a}_{(n)},$ as $q\rightarrow\infty$.

Hast and Matei in \cite[Thm 4.4]{HaMa} have considered a class of functions called arithmetic functions of von Mangoldt type, which is similar to the class of factorization functions as defined here (see \cite{HaMa} for details of the definition). For this class of functions, Hast and Matei prove what may be thought of as a first-order short interval analogue of Andrade, Bary-Soroker, and Rudnick's result in Theorem \ref{thm:factorizationindependence}. Rewritten in the notation used above: 
\begin{thm}
\label{thm:hastmatei}
For a fixed arithmetic function of von Mangoldt type $a(f)$, and fixed $n\geq 4$, $1 \leq h \leq n-3$, and $k\geq 1$, 
\begin{equation}
\label{eq:hastmatei}
\frac{1}{q^n} \sum_{f\in \mathcal{M}_n} \Big(\sum_{g \in I(f;h)} a(g)\Big)^k = q^{k(h+1)} \Big( \mathbb{E}_{f\in\mathcal{M}_n} a(f)\Big)^k + O(q^{(k-1)(h+1)}),
\end{equation}
as $q\rightarrow\infty$.
\end{thm}

In the case $k=2$ this is sufficient to recover the upper bound of $O(q^{h+1})$ for the variance computed by Theorem \ref{thm:b_mangoldtvar} of Keating and Rudnick, though not the constant $n-h-2$. 

We note the approach of Hast and Matei to study short intervals is rather different from ours -- in particular they do not require any of the facts about $L$-functions that we will make use of in what follows. Other related recent papers with a perspective similar to Hast and Matei's, making use of the connection between polynomials over a finite field and the symmetric group to investigate arithmetic functions defined on $\mathcal{M}$, include \cite{ChElFa, Gad}.

\section{A statement of main results}
\label{sec:3}

\subsection{}
We are now in a position to state our main results.

\begin{thm}
	\label{thm:general_var}
	For $a(f)$ a fixed factorization function, and fixed $h$ and $n$ with $0\leq h \leq n-5,$
	\vspace{1mm}
	\begin{equation}
		\label{eq:general_var}
		\Var_{f\in \mathcal{M}_n}\Big( \sum_{g \in I(f;h)} a(g)\Big) = q^{h+1} \sum_{ \substack{\lambda \vdash n \\ \lambda_1 \leq n-h-2}} |\hat{a}_\lambda|^2 + O(q^{h+1/2}).
	\end{equation}
\end{thm}
Here the coefficients $\hat{a}_\lambda$ are defined by the expansion \eqref{eq:factfourier}, and the sum in \eqref{eq:general_var} is over all partitions $(\lambda_1,\lambda_2,...)$ of $n$ such that $\lambda_1$ (and therefore every $\lambda_i$) is no more than $n-h-2$.

In \eqref{eq:general_var}, the implied constant of the error term depends on $h$, $n$ and the factorization function itself, so the result is only of interest as $q\rightarrow\infty$. 

In section \ref{sec:fact_exp} we compute the coefficients in the expansion \eqref{eq:factfourier} for the factorization functions $\mu(f)$, $\Lambda(f)$, $\Lambda_j(f)$ and $d_k(f)$. These expansions, applied in Theorem \ref{thm:general_var} are sufficient to recover estimates for the variance of sums of these arithmetic functions over short intervals which we have cited in Theorems \ref{thm:a_mobiusvar}, \ref{thm:b_mangoldtvar}, and \ref{thm:f_dvar}.


It has been observed that in each of these Theorems from section \ref{sec:1}, the main term of the variance works out to be an integer, which suggests perhaps that in computing the variance a certain object is being counted. One simple consequence of Theorem \ref{thm:general_var} is that this is not the case in general; it is easy to find factorization function $a: \mathcal{M} \rightarrow \mathbb{Z}$ for which $\hat{a}_\lambda$ is not always an integer and for which the variance is non-integer. An example is furnished by the arithmetic functions $\omega(f)$, counting the number of prime factors of $f$, and likewise the function $\mu(f)\omega(f)$. We consider the short interval variance of these functions in section \ref{sec:fact_exp} by using Theorem \ref{thm:general_var}; this leads us to make a conjecture in the setting of the integers which seems perhaps somewhat surprising.



Note also that Theorem \ref{thm:general_var} gives us a non-trivial upper bound for the variance of arithmetic functions supported on the squarefrees, though the upper bound is one which may be far from optimal. Work of Keating and Rudnick \cite{KeRu2} and Roditty-Gershon \cite{RoGe} considers some related questions about the squarefrees (and indeed square-fulls) more carefully to get asymptotics, not only upper bounds.

The variance evaluation in Theorem \ref{thm:general_var} comes in part from a combinatorial analysis of random matrix integrals. In particular the already mentioned function field equidistribution theorem of Katz plays an important role in the proof. 

A likewise central role is played by a combinatorial analogue of the explicit formula of Weil, relating the zeros of an $L$-function to certain arithmetic functions. In particular, in section \ref{sec:schur} and especially Theorem \ref{thm:schur_in_zeros} we show that Schur functions of zeros of $L$-functions are closely related to the characters $X^\lambda(f)$ defined above. 

We note the conjectural appearance of the symmetric group in other closely related contexts, for example in Dehaye's work on moments of the Riemann zeta function \cite{De}. It would be of interest to pursue this connection further.

\subsection{}
The same result may be stated perhaps more strikingly along the lines advertised in section \ref{sec:1}. Let $\mathcal{F}_n$ be the linear space of factorization functions supported on $\mathcal{M}_n$, and define $\mathcal{U}_n^h$ to be the subspace of factorization functions for which variance is negligible; that is,
\begin{equation}
\label{eq:udef}
\mathcal{U}_n^h:= \Big\{ u \in \mathcal{F}_n:\; \lim_{q\rightarrow\infty} \Var_{f\in\mathcal{M}_n}\Big( \sum_{g \in I(f;h)} u(g) \Big) = o(q^{h+1})\Big\}.
\end{equation}
We may endow $\mathcal{F}_n$ with an inner product: for $a_1, a_2 \in \mathcal{F}_n$, we define
\begin{equation}
\label{eq:innerdef}
\langle a_1, a_2 \rangle := \lim_{q\rightarrow\infty} \frac{1}{q^n} \sum_{f\in \mathcal{M}_n} a_1(f) \overline{a_2(f)}.
\end{equation}
This inner product is degenerate, but only on factorization functions supported on the squarefuls. If we decompose $\mathcal{F}_n = \mathcal{G}_n\oplus\mathcal{B}_n$, where $\mathcal{G}_n$ is the space of factorization functions supported on squarefree monic polynomials of degree $n$, and $\mathcal{B}_n$ is the space supported on squarefuls, then the equidistribution of factorization types imply that this is a proper inner product when restricted to $\mathcal{G}_n$. 

We will show that $\mathcal{B}_n \subseteq \mathcal{U}_n^h$, and so if we define $\mathcal{V}_n^h$ to be the orthogonal complement to $\mathcal{U}_n^h$ inside $\mathcal{G}_n$, we have
$$
\mathcal{F}_n = \mathcal{U}_n^h \oplus \mathcal{V}_n^h.
$$

We will observe the following restatement of Theorem \ref{thm:general_var},

\begin{thm}
	\label{thm:subspace_decompose1}
	Let $0 \leq h \leq n-5$ be fixed and $v$ be a fixed factorization function from the subspace $\mathcal{V}_n^h$. Then
	$$
	\Var_{f\in\mathcal{M}_n}\bigg( \sum_{g \in I(f;h)} v(g)\bigg) = q^{h+1}\langle v, v \rangle + O(q^{h+1/2}).
	$$
\end{thm}
That is, for $\mathcal{V}_n^h$, only diagonal terms contribute to the variance, while by definition for $\mathcal{U}_n^h$ the variance is of lower order. This implies an estimate for the variance of an arbitrary factorization function $a \in \mathcal{F}_n$, since there is a unique decomposition $a = u+v$ with $u \in \mathcal{U}_n^h$ and $v\in\mathcal{V}_n^h.$ Indeed,
$$
\Var_{f\in\mathcal{M}_n}\bigg( \sum_{g\in I(f;h)} u(g)\bigg) = o(q^{h+1}),
$$
so (using Cauchy-Schwarz to bound covariance),
\begin{equation}
\label{eq:var_project}
\Var_{f\in\mathcal{M}_n}\bigg( \sum_{g \in I(f;h)} a(g)\bigg) = q^{h+1}\langle v, v \rangle + o(q^{h+1}).
\end{equation}

The spaces $\mathcal{U}_n^h$ and $\mathcal{V}_n^h$ can be characterized explicitly.
\begin{prop}
	\label{prop:subspace_decompose2}
	We have $$\mathcal{U}_n^h = \mathcal{A}_n^h \oplus \mathcal{B}_n,$$ where
	\begin{align*}
	\mathcal{A}_n^h&:= \mathrm{span}\{ X^\lambda(f)\; : \; \lambda \vdash n,\, \lambda_1 \geq n -h -1\},\\
	\mathcal{B}_n &:= \{b(f):\; b\in \mathcal{F}_n \textrm{ is supported on squareful elements}\}.
	\end{align*}
	Furthermore
	$$
	\mathcal{V}_n^h = \mathrm{span}\{ X^\lambda(f)\; : \; \lambda \vdash n,\, \lambda_1 \leq n -h -2\}.
	$$
\end{prop}
This explicit decomposition is what connects Theorem \ref{thm:general_var} and Theorem \ref{thm:subspace_decompose1}. It is worthwhile to emphasize once again an interpretation of this result; the determination that the variance of short interval sums of functions lying in $\mathcal{U}_n^h$ is negligible will be a relatively simple fact to verify -- we will that functions lying in this space are forced to be regular across short intervals owing in the end to a paucity of zeros of $L$-functions. The theorem tells us that outside this first obstruction, factorization functions otherwise behave in an oscillatory fashion, akin to the M\"obius function, when summed in a short interval.

There is another appealing way to write this decomposition, based on a suggestion by J. Ellenberg:

\begin{prop}
	\label{prop:subspace_decomposedivisors}
	Define the space $\mathcal{U}_n^h$ as in the start of this subsection. Then $\mathcal{U}_n^h$ consists of functions $u(f)$ that can be written in the form
	\begin{equation}
	\label{divisor_decompose}
	u(f) = \sum_{\substack{\delta | f \\ \deg(\delta) \leq h+1}} \alpha(\delta) + b(f), \quad \textrm{for all}\; f\in \mathcal{M}_n,
	\end{equation}
	where $\alpha(\delta)$ is a factorization function, and $b(f)$ is a factorization function supported on the squarefuls.
\end{prop}

Here the sum is over all monic polynomials $\delta$ dividing $f$ with degree no more than $h+1$.

Indeed, it will again follow quite easily that for all factorization functions that can be represented as truncated divisor sums in this way, the value of their sums over short intervals will remain basically constant no matter the choice of short interval, so that these sums have negligible variance. The space $\mathcal{V}_n^h$ remains defined as the complement of $\mathcal{U}_n^h$, and so an interpretation of this decomposition remains the same -- outside an `easy-to-find' obstruction, functions otherwise behave in an oscillatory fashion when summed in a short interval.

As a corollary of Theorem \ref{thm:subspace_decompose1} and Proposition \ref{prop:subspace_decomposedivisors}, we have

\begin{cor}
\label{cor:variance_to_minimizer}
For $a(f)$ a fixed factorization function, and fixed $h$ and $n$ with $0 \leq h \leq n-5$,
\begin{equation}
\label{eq:covariance_to_minimizer}
\Var_{f\in\mathcal{M}_n}\bigg( \sum_{g \in I(f;h)} a(g)\bigg) = q^{h+1} \inf_{\alpha \in \mathcal{F}} \Big\| a(f) - \sum_{\substack{\delta | f \\ \deg(\delta) \leq h+1}} \alpha(\delta) \Big\|^2 + O(q^{h+1/2}),
\end{equation}
where $\mathcal{F}$ is the space of all factorization functions, and $\|\cdot\|$ is the norm induced by the inner product \eqref{eq:innerdef}.
\end{cor}

Rather curiously, the minimization problem arising in the computing the right hand side of \eqref{eq:covariance_to_minimizer} has some similarity to those which arise in connection to the Selberg sieve.

We turn to a proof of these decompositions and Theorem \ref{thm:subspace_decompose1} in section \ref{sec:10.1}.

\subsection{}
\label{sec:intro4.5}
Because Theorem \ref{thm:general_var} allows us to compute variances for general factorization functions, it is straightforward to also compute covariances using it. We record a general formula for covariance in section \ref{sec:covariance}, and draw out some interesting consequences that appear to be new in the literature.

\subsection{}
\label{sec:intro5}
A similar set of results could be developed for factorization functions in arithmetic progressions rather than short intervals, though we don't do so here.

\subsection{}
\label{sec:intro6}
In the next two sections we recall some background material regarding Dirichlet characters, L-functions, and symmetric function theory. We turn to the substantial portion of the proof of Theorem \ref{thm:general_var} in section \ref{sec:short_interval}.

\section{Background on Dirichlet characters and zeros of L-functions}
\label{sec:dirichlet}
\subsection{}
\label{sec:dirichlet1}
We recall a few of the basic facts about Dirichlet characters defined over $\mathbb{F}_q[T]$ that we will use. Our notation is the same as that from \cite{Ros,KeRu1,Ro,Ru,KeRoRoRu} and a reader familiar with the facts from any one of those may skip this section and refer back to it as it is referenced.


In $\mathbb{F}_q[T]$, we will make use of the family of primitive even characters modulo the element $T^M$ for powers $M\geq 1$. We call a character $\chi$ \emph{even} if for all $c \in \mathbb{F}_q$ and all $f\in\mathbb{F}_q[T]$, we have $\chi(cf) = \chi(f)$. Recall that the number of  Dirichlet characters modulo $T^M$ is
\begin{equation}
	\label{eq:phi_eval}
	\Phi(T^M) = q^M(1-1/q),
\end{equation}
the number of primitive Dirichlet characters is
\begin{equation}
	\label{eq:phiprim_eval}
	\Phi_{prim}(T^M) = q^M(1-1/q)^2,
\end{equation}
the number of even Dirichlet characters is
\begin{equation}
	\label{eq:phiev_eval}
	\Phi^{ev}(T^M) = q^{M-1},
\end{equation}
and the number of even primitive characters is 
\begin{equation}
	\label{eq:phievprim_eval}
	\Phi_{prim}^{ev}(T^M) = q^{M-1}(1-1/q).
\end{equation}

We recall that the L-function of a Dirichlet character $\chi$ is defined for $|u| < 1/q$ by
$$
\mathcal{L}(u,\chi):= \sum_{f\, \mathrm{monic}} \chi(f)u^{\deg(f)} = \prod_{\substack{P \; \mathrm{monic,}\\\mathrm{irred.}}}\frac{1}{1-\chi(P) u^{\deg(P)}},
$$
and that for $\chi$ non-trivial that $\mathcal{L}(u,\chi)$ is a polynomial in $u$, defined for $|u| \geq 1/q$ by analytic continuation. The Riemann hypothesis, in this context a theorem of Weil \cite{We}, states that all roots of $\mathcal{L}(u,\chi)$ lie on the circles $|u| = q^{-1/2}$ or $|u|=1$. If $\chi$ is a non-trivial character modulo a polynomial $Q$ of degree $M$, then $\mathcal{L}(u,\chi)$ has no more than $M-1$ roots, and as a well known consequence of this and the Riemann hypothesis,
\begin{equation}
	\label{vonMangoldtBound}
	\sum_{f\in \mathcal{M}_n} \Lambda(f)\chi(f) = O_M(q^{n/2}).
\end{equation}

\subsection{}
\label{sec:dirichlet2}
In the case that $\chi$ is a primitive character we can succinctly say more. In this case for $\chi$ modulo $T^M$, the polynomial $\mathcal{L}(u,\chi)$ has exactly $M-1$ roots. Define the function $\lambda_\chi$ to be $1$ if $\chi$ is even, and $0$ otherwise. When $\chi$ is even, $\mathcal{L}(u,\chi)$ has a simple zero at $u=1$, otherwise all zeros of this polynomial lie on the circle $|u| = q^{-1/2}$. We can record this information in a single equation; we have for primitive characters $\chi$,
\begin{align}
	\label{frob}
	\notag \mathcal{L}(u,\chi) &= (1-\lambda_\chi u) \prod_{j=1}^N(1-q^{1/2} e^{i 2\pi \vartheta_j} u) \qquad \text{ for } N:= \deg Q - 1 - \lambda_\chi \\
	&= (1 -\lambda_\chi u) \det(1 - q^{1/2} u \,\Theta_\chi), 
\end{align}
where $e^{i2\pi\vartheta_1}, ..., e^{i2\pi \vartheta_N}$ lie the unit circle and are determined by the character $\chi$, and 
$$
\Theta_\chi := \textrm{diag}(e^{i2\pi\vartheta_1}, ..., e^{i2\pi \vartheta_N}).
$$
is known as the unitarized Frobenius matrix. From logarithmic differentiation we also have the \emph{explicit formula},
\begin{equation}
	\label{explicit}
	\sum_{f\in \mathcal{M}_n} \Lambda(f) \chi(f) = -q^{n/2}\, \Tr\, \Theta_\chi^n - \lambda_\chi.
\end{equation}

To control the distribution of zeros, a theorem of Katz's will be important for us, as it has been in all investigations of this sort since Keating and Rudnick's \cite{KeRu1}.
We let $PU(m)$ be the projective unitary group, the quotient of the unitary group $U(m)$ by unit modulus scalars, endowed with Haar measure, and $PU(m)^{\#}$ be the space of conjugacy classes of $PU(m)$, with inherited measure.

\begin{thm}[Katz \cite{Ka}]
	\label{katz}
	Fix $M\geq 5$. Over the family of even primitive characters $\chi \MOD T^{M})$, the conjugacy classes of the unitarized Frobenii $\Theta_\chi$ become equidistributed in $PU(M-2)^{\#}$ as $q\rightarrow\infty$.
\end{thm}

More computationally the meaning of the theorem is as follows: for any continuous class function $\phi: U(M-2) \rightarrow \mathbb{C}$ such that $\phi(e^{i2\pi \theta} g) = \phi(g)$ for all unit scalars $e^{i2\pi \theta}$ and unitary matrices $g$, we have
$$
\lim_{q\rightarrow\infty}\mathbb{E}_{\substack{\chi\, (T^M) \\ \text{prim., ev.}}} \phi(\Theta_\chi) =  \int_{U(M-2)} \phi(g)\,dg,
$$
as $q\rightarrow\infty$, where for typographical reasons we have written
$$
\mathbb{E}_{\substack{\chi\, (T^M) \\ \text{prim., ev.}}} \phi(\Theta_\chi) := \frac{1}{\Phi_{prim}^{ev}(T^{M})} \sum_{\substack{\chi\, (T^{M}) \\ \text{prim., ev.}}} \phi(\Theta_\chi). 
$$

\subsection{}
\label{sec:dirichlet3}
The reason we will be interested in characters modulo $T^M$ is the following involution used by Keating and Rudnick. 

We let $\mathcal{P}_n$ be the collection of degree $n$ polynomials in $\mathbb{F}_q[T]$, and $\mathcal{P}_n^\natural := \{f \in \mathcal{P}_n:\; (f,T)=1\}$. Equivalently $\mathcal{P}_n^\natural$ is the collection of degree $n$ polynomials with a constant coefficient that is non-zero. Our involution is the mapping $f\mapsto f^\ast$ from $\mathcal{P}_n^\natural$ to itself defined by
\begin{equation}
	\label{eq:invol}
	(a_0 + a_1 T^1 + \cdots + a_n T^n)^\ast = a_n + a_{n-1} T + \cdots + a_0 T^n.
\end{equation}
It is straightforward to check that for $f$ with non-zero constant coefficient,
$$
(f^\ast)^\ast = f,
$$
and for $f,g$ with non-zero constant coefficient,
$$
(fg)^\ast = f^\ast g^\ast.
$$
If we extend the definition of factorization type to $\mathcal{P}_n$, so that for $f\in \mathcal{P}_n$ for that scalar $c\in \mathbb{F}_q$ such that $cf \in \mathcal{M}_n$, the factorization type of $f$ is defined to be the factorization type of $cf$, it follows that for $f\in \mathcal{P}_n^\natural$,
\begin{equation}
	\label{eq:fact_invol}
	\tau_f = \tau_{f^\ast}.
\end{equation}

This involution is useful for us because for $g_1,g_2 \in \mathcal{P}_n^\natural$, 
\begin{equation}
	\label{eq:distance_invol}
	\notag \deg(g_1-g_2) \leq h
\end{equation}
if any only if
\begin{equation}
	\label{eq:cong_invol}
	\notag g_1^\ast - g_2^\ast \equiv 0 \MOD T^{n-h}).
\end{equation}
This equivalence is easily checked. It is because of this that we may use Dirichlet characters and their L-functions to study short interval sums.

\section{Background on symmetric function theory}
\label{sec:sym}
\subsection{}
\label{sec:sym1}
We recall some notation and well known facts from symmetric function theory that we will use in what follows. A standard reference and introduction to the material we recall here is \cite[Ch. 7]{St}.

We have already defined partitions and discussed their basic notation in section \ref{sec:intro2}. One additional way to represent partitions is as a Young diagram. This is an array of left-justified boxes, with the number of boxes in each rowing weakly decreasing. For a partition $\lambda$, the Young diagram corresponding to $\lambda$ has $\lambda_1$ boxes in its first row, $\lambda_2$ boxes in its second row, etc. For instance, the Young diagram with shape $(5,3,3,1)$ is as follows:
$$
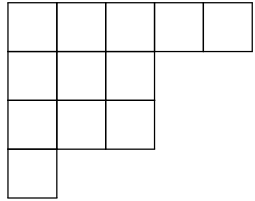
\ytableausetup{mathmode, boxsize=1.5em}
\ytableaushort{\none,\none,\none,\none}*{5,3,3,1}
$$

The \emph{dual partition} $\lambda'$ is defined to be $(\lambda_1',\lambda_2',...)$ where $\lambda_i'$ is the number of boxes in the $i$-th column of the Young diagram corresponding to $\lambda$. So in our example above, $(5,3,3,1)' = (4,3,3,1,1)$.

The \emph{length of a partition}, $\ell(\lambda)$ is defined to be $k$, where $\lambda = (\lambda_1,...,\lambda_k)$. So for instance $\ell(5,3,3,1) = 4$.

Young diagrams may be used to write down a relatively simple expression for characters of the symmetric group in the form of the famous Murnaghan-Nakayama rule. We quickly recall it here, taking from the presentation in \cite[Sec. 7.17]{St}, recommended for those who have not seen this result before. As a prerequisite, we define \emph{Young tableaus of shape $\lambda$} to be arrays of numbers, weakly increasing across rows and down columns, written in the squares of a Young diagram of $\lambda$. A \emph{border strip tableau of shape $\lambda$ and type $\tau$} is a Young tableau such that among the entries the number $i$ occurs exactly $\tau_i$ times, and for each $i$ the set of squares in which $i$ have been written form a \emph{border strip} -- that is, a connected collection of squares with no square upward and to the left of any others. The \emph{height} of a border strip is one less than the number of rows that contain it, and the height $h(T)$ of a tableau $T$ composed of border strips is the sum of the heights of the border strips.

\begin{thm}[Murnaghan-Nakayama rule]
\label{thm:murnaghan}
	For $\lambda$ a partition of $n$, and $\tau$ the type of a permutation from $\mathfrak{S}_n$
	\begin{equation}
	\label{eq:murnaghan}
	X^\lambda(\tau) = \sum_T (-1)^{h(T)},
	\end{equation}
	where the sum is over all border strip tableaus $T$ of shape $\lambda$ and type $\tau$.
\end{thm}

\textbf{Remark:} A reader unfamiliar with characters of the symmetric group but nonetheless comfortable with the statement of the Murnaghan-Nakayama rule may take \eqref{eq:murnaghan} as their definition the symmetric group's characters.

\subsection{}
\label{sec:sym2}
We will need to work with symmetric polynomials in $m$ variables. Two bases for these polynomials that will be important for us are the power sum symmetric functions and Schur functions. Both bases are indexed by partitions. 

For \emph{power sum symmetric functions} in the variables $\omega_1,...,\omega_m$ we recall the definition that for an integer $n$,
$$
p_n = p_n(\omega_1,...,\omega_m) := \omega_1^n + \cdots + \omega_m^n,
$$
and for a partition $\lambda = (\lambda_1,...,\lambda_k)$, we define
$$
p_\lambda := p_{\lambda_1} \cdots p_{\lambda_k}.
$$
It is an elementary fact \cite[Corollary 7.7.2]{St} that any symmetric polynomial in the variables $\omega_1,..,\omega_m$ can be expressed uniquely as a linear combination of the functions $p_\lambda$.

\emph{Schur functions} in the variables $\omega_1,...,\omega_m$ have the following as their classical definition. For a partition $\lambda$ with $\ell(\lambda)\leq m$, set
$$
s_\lambda = s_\lambda(\omega_1,...\omega_m) := \frac{\det\Big(\omega_i^{\lambda_j+m-j}\Big)_{i,j=1}^m}{\det\Big(\omega_i^{n-j}\Big)_{i,j=1}^m}.
$$
If $\ell(\lambda) < m$, we extend $\lambda$ with $0$'s in the extra places so that the above definition still makes sense -- i.e. $\lambda = (\lambda_1,...,\lambda_k,0,...0)$. If $\ell(\lambda) > m$, we set $s_\lambda = 0.$

It is well known (though not completely obvious at first glance) that $s_\lambda$ defined as above is a symmetric polynomial with integer coefficients. As with power sums, any symmetric polynomial in the variables $\omega_1,...,\omega_m$ can be expressed uniquely as a linear combination of the functions $s_\lambda$. Proofs of these facts may be found in \cite[Ch. 7]{St}.

For these symmetric polynomials we have the following important identities:


\begin{thm}[Frobenius]
	\label{thm:frobenius}
	For $\lambda \vdash n$,
	\begin{align}
		\label{eq:frobenius}
		\notag s_\lambda &= \frac{1}{n!} \sum_{\sigma \in \mathfrak{S}_n} X^\lambda(\sigma)  \,p_\sigma \\
		&= \sum_{\nu\, \vdash n} \mathbf{p}(\nu) X^\lambda(\nu) p_\nu.
	\end{align}
	
	Likewise,
	\begin{equation}
		\label{eq:dual_frobenius}
		p_\nu = \sum_{\lambda \vdash n} X^\lambda(\nu) s_\lambda.
	\end{equation}
\end{thm}

\begin{proof}
	\eqref{eq:frobenius} is Theorem 7.17.3 of \cite{St}, while \eqref{eq:dual_frobenius} is Corollary 7.17.4.
\end{proof}

We can also express $s_\lambda$ in terms of the elementary symmetric functions, defined by 
$$
e_n = e_n(\omega_1,...,\omega_m) := \sum_{i_1 < \cdots < i_n} \omega_{i_1}\cdots \omega_{i_n},
$$
with the conventions $e_0 = 1$ and $e_n(\omega_1,...,\omega_m) = 0$ for $n > m$.

\begin{thm}[Jacobi-Trudi]
	\label{thm:jacobi_trudi}
	For $\lambda_1 \leq k$,
	$$
	s_\lambda = \det\big( e_{\lambda'_i -i + j} \big)_{i,j=1}^k.
	$$
\end{thm}

\begin{proof}
	This is a special case of Corollary 7.16.2 of \cite{St}.
\end{proof}

\textbf{Remark:} This is often known as the \emph{dual} Jacobi-Trudi identity because there is an equivalent formula in terms of the complete homogeneous symmetric functions; see \cite[Thm 7.16.1]{St}.

\subsection{}
\label{sec:sym3}
One of the many results that is derived in the literature from Theorem \ref{thm:frobenius} is an identity for characters of the symmetric group indexed by partitions that are dual to each other. We cite it here because we will use it later.
\begin{prop}
	\label{prop:dual_part}
	For $\sigma \in \mathfrak{S}_n$ and $\lambda \vdash n$,
	$$
	X^{\lambda'}(\sigma) = (-1)^{n-\ell(\sigma)} X^\lambda(\sigma).
	$$
\end{prop}
\noindent Here $\ell(\sigma) := \ell(\tau_\sigma)$.

\begin{proof}
	This is example 2 of section I.7 in \cite{Ma}. 
\end{proof}

\subsection{}
\label{sec:sym4}
One of the reasons we are interested in Schur functions is their appearance in random matrix theory. It is well known that they satisfy the following orthogonality relation.

\begin{thm}
	\label{thm:schur_ortho}
	For partitions $\lambda, \nu$,
	$$
	\int_{U(m)} s_\lambda(g) \overline{s_\nu(g)}\, dg = \delta_{\lambda \nu} \cdot \delta_{\ell(\lambda), \ell(\nu) \leq m}.
	$$
	Moreover, if $\lambda$ and $\nu$ are partitions of the same number (that is  $|\lambda| =  |\nu|$)
	$$
	\int_{PU(m)} s_\lambda(g) \overline{s_\nu(g)}\, dg = \delta_{\lambda\nu} \cdot \delta_{\ell(\lambda), \ell(\nu) \leq m}.
	$$
\end{thm}
\noindent Here $s_\lambda(g), s_\nu(g)$ are Schur functions whose entries are the $m$ eigenvalues of the matrix $g$. A more or less self-contained proof may be found in \cite{Ga} as well as in more standard texts on representation theory.

\section{A basis for factorization functions, and a bound for character sums}

\subsection{}
\label{sec:short_interval1}
We turn in this section to a proof of Theorem \ref{thm:general_var}. Out strategy will be a familiar one, similar in its broad outlines to the original proof of Keating and Rudnick. By making use of the involution described in section \ref{sec:dirichlet}, we transfer a short interval sum to an average over sums of Dirichlet characters against factorization functions. These are in turn evaluated by using an equidistribution result of Katz and the combinatorial analysis of section \ref{sec:schur}. This combinatorial analysis is perhaps the most important observation of the paper. In terms of technique, some new issues arise that have not appeared in the past just because we work with factorization functions in general.

\subsection{}
\label{sec:short_interval2}
We begin by noting some ways to build factorization functions out of simpler functions. For two arithmetic functions $\phi_1$ and $\phi_2$ we define the convolution in the usual way,
$$
\phi_1\star \phi_2(f) := \sum_{\substack{f_1 f_2 = f \\ f_1, f_2 \in \mathcal{M}}} \phi_1(f) \phi_2(f).
$$
It is clear that if $\phi_1$ and $\phi_2$ are factorization functions, then $\phi_1\star \phi_2$ will be a factorization function as well.

For integers $m, e \geq 1$ we define the factorization function
$$
\iota_{m,e}(f) = \begin{cases} 1 & \textrm{if}\; f = P^e\; \textrm{with}\; \deg(P) = m \\
0 & \textrm{otherwise}
\end{cases}
$$
Thus $\iota_{m,e}$ is the indicator function of $e^\textrm{th}$ powers of $m^\textrm{th}$ degree primes, and is supported on $\mathcal{M}_{me}$. We generalize it in the following way: for an array $(\mathbf{m}, \mathbf{e}) = (m_1, e_1; m_2, e_2; ...; m_\ell, e_\ell)$ we define
\begin{equation}
\label{eq:indicator_def}
\iota_{(\mathbf{m},\mathbf{e})} = \iota_{m_1,e_1}\star \iota_{m_2,e_2} \star \cdots \star \iota_{m_\ell, e_\ell}.
\end{equation}

\begin{prop}
	\label{prop:alg_span}
	Any factorization function supported on $\mathcal{M}_n$ is a linear combination of the functions $\iota_{(\mathbf{m},\mathbf{e})}$. (Necessarily $m_1 e_1 + \cdots + m_\ell e_\ell = n$).
\end{prop}

\begin{proof} 
	Let $\mathcal{M}_{n,L}$ be the collection of elements of $\mathcal{M}_n$ with extended factorization type $(m_1, e_1; \cdots; m_\ell, e_\ell)$ for $\ell \leq L$. We suppose the Proposition is true for all factorization functions supported on $\mathcal{M}_{r,L}$ with $r\leq n$ and show that it is true for $\mathcal{M}_{n,L+1}$. Since it is obviously true (check) for $\mathcal{M}_{n,1}$ for all $n$, this will verify the claim by induction.
	
	We introduce indicator functions $I_{(\mathbf{m},\mathbf{e})}$ of the extended factorization type $(\mathbf{m},\mathbf{e})$; that is for $f\in\mathcal{M}$, we set $I_{(m_1,e_1;...;m_\ell,e_\ell)}(f) = 1$ if $f$ has extended factorization type $(m_1,e_1;...; m_\ell;e_\ell)$ and $I_{(m_1,e_1;...;m_\ell,e_\ell)}(f) = 0$ otherwise. Clearly
	$$
	\mathcal{M}_{n,L+1} = \mathrm{span}\{I_{(m_1,e_1;...;m_\ell,e_\ell)}:\; m_1 e_1 + \cdots + m_\ell e_\ell = n,\; \ell \leq L+1\},
	$$
	so to prove our claim we need only show that each 
	\begin{equation}
	\label{eq:I_labeled}
	I_{(m_1,e_1;\cdots; m_{L+1},e_{L+1})},
	\end{equation}
	is a linear combination of functions $\iota_{(\mathbf{m},\mathbf{e})}$. Suppose $\nu$ of the terms $(m_1,e_1), ..., (m_L,e_L)$ in \eqref{eq:I_labeled} are equal to $(m_{L+1}, e_{L+1})$. (We allow $\nu$ to be $0$.) By inspection of elements of $\mathcal{M}_{n,L+1}$ we see that
	\begin{equation}
	\label{eq:I_subtract}
	I_{(m_1,e_1;...;m_L,e_L)}\star \iota_{m_{L+1},e_{L+1}} - (\nu+1) I_{(m_1,e_1;...; m_{L+1},e_{L+1})}
	\end{equation}
	is supported on $\mathcal{M}_{n,L}$. By inductive hypothesis then \eqref{eq:I_subtract} is a linear combination of terms $\iota_{(\mathbf{m},\mathbf{e})}$. Likewise by inductive hypothesis, $I_{(m_1,e_1;...,m_L,e_L)}$ is a linear combination of such terms, so $I_{(m_1,e_1;...;m_L,e_L)}\star \iota_{m_{L+1},e_{L+1}}$ will be as well. Returning to \eqref{eq:I_subtract}, since $\nu + 1 \neq 0$, this shows that $I_{(m_1,e_1;...,m_{L+1},e_{L+1})}$ is therefore a linear combination of such terms, so that as claimed all factorization functions on $\mathcal{M}_{n,L+1}$ are linear combinations of such terms also.
\end{proof}

\subsection{}
\label{sec:short_interval2.5}
We have indicated that we must work with Dirichlet characters modulo $T^M$ for some power $M$. Note that for any non-trivial Dirichlet character $\chi$ modulo $T^M$, we have, by excluding powers of primes from the sum in the first line below and using the Riemann Hypothesis in the form \eqref{vonMangoldtBound} in the second,
\begin{align*}
	\sum_{f\in \mathcal{M}_n} \iota_{n,1}(f) \chi(f) &= \frac{1}{n} \sum_{f\in \mathcal{M}_n} \Lambda(f) \chi(f) + O(q^{n/2}) \\
	&= O_M(q^{n/2}).
\end{align*}
Thus for any $e \geq 2$, as long as $\chi^e \neq \chi_0$,
$$
\sum_{f\in\mathcal{M}_{me}} \iota_{m,e}(f) \chi(f) = \sum_{f\in \mathcal{M}_m} \iota_{m,1}(f) \chi^e(f) = O_M(q^{m/2}).
$$
For $e \geq 3$, trivially
$$
\sum_{f\in \mathcal{M}_{me}} \iota_{m,e}(f) \chi(f) = O(q^m).
$$
Note that for $m\geq 1, e\geq 2$, we have $m/2 \leq me/2-1/2$, and for $m\geq 1, e\geq 3,$ we have likewise $m\leq me/2-1/2$. Thus combining the two estimate above, we see that unless $\chi^2 = \chi_0$, we have
$$
\sum_{f\in \mathcal{M}_{me}} \iota_{m,e}(f) \chi(f) = O(q^{me/2-1/2}).
$$
Hence recalling the definition \eqref{eq:indicator_def}, unless $\chi^2 = \chi_0$, if any $e_i \geq 2$,
\begin{equation}
	\label{eq:squarefulind_bound}
	\sum_{f\in\mathcal{M}_n} \iota_{\mathbf{m},\mathbf{e}}(f) \chi(f) = O_{M,n}(q^{n/2-1/2}),
\end{equation}
where $n = m_1 e_1 + \cdots + m_k e_k.$

We have thus obtained
\begin{lem}
	\label{lem:squareful_bound}
	If $b$ is a fixed factorization function supported on the squarefuls, for $\chi$ a Dirichlet character modulo $T^M$, as long as $\chi^2 \neq \chi_0$,
	$$
	\sum_{f \in \mathcal{M}_n} b(f) \chi(f) = O_{M,n}(q^{n/2-1/2}).
	$$
\end{lem}

\begin{proof}
	For such $b$, the function $b(f)\mathbf{1}_{\mathcal{M}_n}(f)$ is necessarily a linear combination of function $\iota_{\mathbf{m},\mathbf{e}}$, with in each case some $e_i \geq 2$.
\end{proof}

In the case that $\chi^2 = \chi_0$, we may genuinely have a worse bound; but it is easy to see in the same way that as long as $\chi \neq \chi_0$ for $\chi \MOD T^M)$, the bound in Lemma \ref{lem:squareful_bound} may be replaced by $O_{M,n}(q^{n/2})$. Indeed, for such an estimate, it is easy to see that we have no need that our factorization function be supported on the squarefuls as it was in Lemma \ref{lem:squareful_bound}:
\begin{lem}
	\label{lem:squareful_free_bound}
	If $a$ is a fixed factorization function, for $\chi$ a non-trivial Dirichlet character modulo $T^M$,
	$$
	\sum_{f\in\mathcal{M}_n} a(f) \chi(f) = O_{n,M}(q^{n/2}).
	$$
\end{lem}

Note that a character satisfies $\chi^2 = \chi_0$ only if it is real. Fortunately there are not many real characters modulo $T^M$:
\begin{lem}
	\label{lem:realchar_bound}
	Over $\mathbb{F}_q[T]$, the number of non-trivial real characters modulo $T^M$ is $O(1)$ if $2\nmid q$, and $O(q^{\lfloor M/2 \rfloor})$ if $2|q$.
\end{lem}

\begin{proof}
	Let $\widehat{G}$ be the group of characters. Real characters $\chi$ are characterized by having $\chi^2 = \chi_0$. As $\widehat{G} \cong (\mathbb{F}_q[T]/(T^M))^\ast$, the number of real characters is equal to the number of $f\in\mathbb{F}_q[T]$ with $(f,T^M)=1$ and $\deg(f) < M$ such that
	\begin{equation}
		\label{eq:realcharcount}
		f^2 \equiv 1 \MOD T^M).
	\end{equation}
	Yet if $2\nmid q$, we have $(f-1,f+1)=1$ and so \eqref{eq:realcharcount} implies $f \equiv \pm 1 \MOD T^M),$ which is satisfied by only two such $f$. Hence in this case there are at most two real characters modulo $T^M$, and thus at most one non-trivial real character.
	
	If $2|q$, the situation is more complicated. If $f = a_0 + \cdots + a_{M-1}$, we have
	$$
	f^2 = a_0^2 + a_1^2 T^2 + \cdots + a_{M-1}^2 T^{2(M-1)},
	$$
	so that each solution $f^2 \equiv 1 \MOD T^M)$ entails $\lfloor (M-1)/2 \rfloor+1$ linear equations,
	$$
	a_0^2 =1, \quad a_1^2 = 0,\quad \cdots \quad , a_{\lfloor (M-1)/2 \rfloor}^2 =0
	$$
	of which there is only one solution. The remaining $M-1 - \lfloor (M-1)/2 \rfloor= \lfloor M/2 \rfloor$ coefficients $a_{\lfloor (M-1)/2 \rfloor+1},...,a_{M-1}$ may vary freely, but this leads to only $q^{\lfloor M/2 \rfloor}$ different solutions.
\end{proof}

\textbf{Remark:} I thank Ofir Gorodetsky for suggesting this proof of Lemma \ref{lem:realchar_bound} to me.

\section{Schur functions of zeros}
\label{sec:schur}
\subsection{}
\label{sec:shur1}
We have noted the explicit formula \eqref{explicit}, which establishes a correspondence between the von Mangoldt function $\Lambda(f)$ and the traces of powers of unitarized Frobenius matrices. Written another way, let $\chi$ be a primitive character modulo $T^m$. For $p_n(\Theta_\chi)$ the symmetric power sum of the unitarized zeros $\{e^{i2\pi \vartheta_1},...,e^{i2\pi\vartheta_{m-2}}\}$ of $\mathcal{L}(u,\chi)$, the explicit formula is just the statement that,
\begin{align}
	\label{eq:powersum}
	\notag p_n(\Theta_\chi) =& \frac{-1}{q^{n/2}} \sum_{f\in \mathcal{M}_n} \Lambda(f) \chi(f) + O(1/q^{n/2}) \\
	=& \frac{-n}{q^{n/2}} \sum_{\substack{P \in \mathcal{M}_n\\\mathrm{prime}}} \chi(P) + O(q^{-1/2})
\end{align}
for $\chi^2 \neq \chi_0$. (We require $\chi^2 \neq \chi$ in order to absorb higher prime powers into the error term.) By multiplying these power sums together, from unique factorization and a simple counting argument, it follows that for the partition $\nu = \langle 1^{m_1} 2^{m_2} \cdots j^{m_j}\rangle$, with $\nu \vdash n$,
$$
p_\nu(\Theta_\chi) = \frac{1}{q^{n/2}} \prod_{i=1}^j i^{m_i} m_i! \sum_{f\in \mathcal{M}_n} \mathbf{1}_\nu(\tau_f)\, \mu(f) \chi(f) + O(q^{-1/2}).
$$
We have used the Riemann hypothesis bound \eqref{vonMangoldtBound} to retain this error term from \eqref{eq:powersum}. Note that the coefficient $\prod i^{m_i} m_i!$ here is $1/\mathbf{p}(\nu)$, defined in equation \eqref{eq:cycleprob} from our introductory remarks about the symmetric group. By applying the Frobenius formula, theorem \ref{thm:frobenius}, we see that for the Schur function with arguments $\{e^{i2\pi \vartheta_1},...,e^{i2\pi\vartheta_{m-2}}\}$,
$$
s_\lambda(\Theta_\chi) = \frac{1}{q^{n/2}} \sum_{f\in\mathcal{M}_n} \mu(f) X^{\lambda}(f)\chi(f) + O(q^{-1/2}).
$$
Because $\mu(f)X^{\lambda}(f) = (-1)^{\ell(\tau_f)}X^\lambda(\tau_f) = (-1)^n X^{\lambda'}(\tau_f)$ by Proposition \ref{prop:dual_part}, we have thus shown,
\begin{thm}
	\label{thm:schur_in_zeros}
	For $\chi$ a primitive character modulo $T^m$ with $\chi^2 \neq \chi_0$,
	$$
	s_\lambda(\Theta_\chi) = \frac{(-1)^n}{q^{n/2}} \sum_{f\in\mathcal{M}_n} X^{\lambda'}(f)\chi(f) + O_{n,m}(q^{-1/2}).
	$$
\end{thm}

\subsection{}
\label{sec:schur2}
Note that in the above theorem, there is no explicit reference to the degree $m$ of the polynomial $T^m$. Nonetheless, if $\chi$ is primitive and even, $s_\lambda(\Theta_\chi)$ is a polynomial in $m-2$ variables, and so we must have $s_\lambda(\Theta_\chi) = 0$ for $\ell(\lambda) > m-2$. We have thus observed

\begin{cor}
	\label{cor:charchar_bound}
	If $\ell(\lambda') = \lambda_1 > m-2$,
	$$
	\sum_{f\in\mathcal{M}_n} X^{\lambda'}(f)\chi(f) = O(q^{(n-1)/2}),
	$$
	uniformly for $\chi$ a primitive even character modulo $T^m$.
\end{cor}

\textbf{Remark:} A similar statement can of course be written down for odd primitive characters.

As another consequence of Theorem \ref{thm:schur_in_zeros},

\begin{cor}
\label{cor:char_ortho}
For partitions $\lambda,\nu \vdash n$ and $m\geq 5$,
\begin{multline}
	\label{eq:char_ortho}
	\mathbb{E}_{\substack{\chi\, (T^m) \\ \mathrm{prim., ev.}}} \Big(\frac{1}{q^{n/2}}\sum_{f\in \mathcal{M}_n} X^{\lambda}(f)\chi(f)\Big) \overline{\Big(\frac{1}{q^{n/2}}\sum_{f\in \mathcal{M}_n} X^{\nu}(f) \chi(f)\Big)} \\ = \delta_{\lambda \nu}\cdot \delta_{\ell(\lambda'),\ell(\nu') \leq m-2} + O(q^{-1/2}).
\end{multline}
\end{cor}

\begin{proof}
	By Theorem \ref{thm:schur_in_zeros}, the left hand side of \eqref{eq:char_ortho} can be written
	\begin{equation}
	\label{eq:char_ortho_reduc1}
	\frac{1}{\Phi_{prim}^{ev}(T^m)} \sum_{\substack{\chi\, (T^m) \\ \mathrm{prim., ev.}}} \Big(s_{\lambda'}(\Theta_\chi) + O(q^{-1/2})\Big) \overline{\Big(s_{\nu'}(\Theta_\chi) + O(q^{-1/2}) \Big)} + O\Big(\frac{q^{\lfloor m/2 \rfloor}}{\Phi_{prim}^{ev}(T^m)}\Big),
	\end{equation}
	using Lemma \ref{lem:squareful_free_bound} and Lemma \ref{lem:realchar_bound} to bound the contribution of characters with $\chi^2 = \chi_0$. For $m\geq 5$, recalling the value of $\Phi_{prim}^{ev}(T^m)$ given in \eqref{eq:phievprim_eval}, we certainly have
	$$
	\frac{q^{\lfloor m/2 \rfloor}}{\Phi_{prim}^{ev}(T^m)} = O(q^{-1/2}),
	$$
	and using the equidistribution Theorem \ref{katz} to treat the main term, we see that \eqref{eq:char_ortho_reduc1} reduces to
	$$
	\int_{U(m-2)} s_{\lambda'} \overline{s_{\nu'}}\, dg + O(q^{-1/2}).
	$$
	This agrees with the right hand side of \eqref{eq:char_ortho} by the orthonormality of Schur functions (Theorem \ref{thm:schur_ortho}).	
\end{proof}

\subsection{}
\label{sec:schur3}
We will later need the following result, which is essentially the `easy' case of Corollary \ref{cor:char_ortho}.
\begin{lem}
\label{lem:diag_by_char} For $a_1$ and $a_2$ factorization functions supported on $\mathcal{M}_n$, and $m$ sufficiently large (depending on $n$),
\begin{equation}
\label{eq:diag_by_char}
\lim_{q\rightarrow\infty} \mathbb{E}_{\substack{\chi\, (T^m) \\ \mathrm{prim., ev.}}} \Big(\frac{1}{q^{n/2}}\sum_{f\in \mathcal{M}_n} a_1(f)\chi(f)\Big) \overline{\Big(\frac{1}{q^{n/2}}\sum_{f\in \mathcal{M}_n} a_2(f) \chi(f)\Big)} = \langle a_1, a_2 \rangle,
\end{equation}
with the inner product defined by \eqref{eq:innerdef}.
\end{lem}

\begin{proof}
This is not a deep result, following from nothing more sophisticated than orthogonality relations for characters averaged in this way. 

Nonetheless, it is less work for us at this point to make use of Corollary \ref{cor:char_ortho} and note the following, for $m \geq \min(5, n+2)$: if $a_1$ or $a_2$ is supported on the squarefuls, then \eqref{eq:diag_by_char} is true (with the right hand side obviously equal to $0$), owing to Lemma \ref{lem:squareful_bound} and Lemma \ref{lem:squareful_free_bound} (with contributions of characters $\chi^2 = \chi_0$ in the average dealt with as in the proof of Corollary \ref{cor:char_ortho}). Moreover, if these functions are characters of the symmetric group, $a_1(f) = X^\lambda(f)$ and $a_2(f) = X^\nu(f)$, then \eqref{eq:diag_by_char} is true by Corollary \ref{cor:char_ortho}. Since any factorization function can be written as a linear combination of characters and some function supported on the squarefuls, this verifies \eqref{eq:diag_by_char} in general.
\end{proof}

\section{A proof of Theorem \ref{thm:general_var}}
\label{sec:short_interval}

\subsection{}
\label{sec:short_interval3}
Because we will be using characters modulo powers of $T$, we must work with polynomials $f$ that are coprime to $T$. We recall our definition $\mathcal{P}_n^\natural$ and make a similar definition for monic polynomials:
$$
\mathcal{P}_n^\natural:= \{f \in \mathcal{P}_n:\, f(0) \neq 0\},\quad \mathcal{M}_n^\natural:= \{ f \in \mathcal{M}_n:\, f(0)\neq 0\}.
$$
In addition we define for $f\in \mathcal{M}_n$,
$$
\widetilde{a}(f):= a(f) - E(a;n), \quad
\mathrm{with}\quad
E(a;n):= \frac{1}{|\mathcal{M}_n|} \sum_{g\in\mathcal{M}_n} a(g)
$$
and for $f \in \mathcal{M}_n^\natural,$
$$
\widetilde{a}^\natural(f) := a(f) - E^\natural(a;n),\quad
\mathrm{with} \quad
E^\natural(a;n):= \frac{1}{|\mathcal{M}_n^\natural|} \sum_{g\in\mathcal{M}_n^\natural} a(g).
$$
With these conventions, our proof of Theorem \ref{thm:general_var} may be broken into five pieces. 

\vspace{2mm}
\emph{\textbf{Step 1:}} In the first place, we reduce the variance of short interval sums, restricted to $\mathcal{M}_n^\natural$, to a sum over Dirichlet characters.

\begin{lem}
	\label{lem:short_to_prog}
	For any factorization function $a$,
	$$
	\sum_{f \in \mathcal{M}_n} \bigg| \sum_{\substack{g \in I(f;h) \\ g \in \mathcal{M}_n^\natural}} \widetilde{a}^\natural(g)\bigg|^2	 = \frac{q^{h+1} (q-1)}{\Phi(T^{n-h})} \sum_{\substack{\chi \neq \chi_0\, (T^{n-h}) \\ \mathrm{even}}} \bigg| \sum_{g\in\mathcal{M}_n} a(g) \chi(g)\bigg|^2.
	$$
	for $0\leq h \leq n$.
\end{lem}
The proof is a straightforward modification of Steps 1 and 2 in \cite{Ro}, and we refer the reader to that paper for details. In summary: one transfers the short interval sum to a sum over Dirichlet characters by making use of the involution described in section \ref{sec:dirichlet} of this paper.

\vspace{2mm}
\emph{\textbf{Step 2:}} We next bound the sums in Lemma \ref{lem:short_to_prog} for all factorization functions that are supported on the squarefuls.
\begin{lem}
	\label{lem:squareful_bound_short}
	For a fixed factorization $b$ function supported on the squarefuls,
	\begin{equation}
		\label{eq:squareful_bound}
		\sum_{f \in \mathcal{M}_n} \bigg| \sum_{\substack{g \in I(f;h) \\ g \in \mathcal{M}_n^\natural}} \widetilde{b}^\natural(g)\bigg|^2 = O_{n,h}(q^{h} q^{n}).
	\end{equation}
	for $0 \leq h \leq n-4$.
\end{lem}

\begin{proof}
	Clearly by Lemma \ref{lem:short_to_prog} we need only show that
	\begin{equation}
	\label{eq:bound_to_show8.2}
	\frac{q-1}{\Phi(T^{n-h})} \sum_{\substack{\chi \neq \chi_0\, (T^{n-h})\\ \mathrm{even}}}\bigg| \sum_{g \in \mathcal{M}_n} b(g) \chi(g)\bigg|^2 = O_{n,h}(q^{n-1}).
	\end{equation}
	From Lemma \ref{lem:squareful_bound}, we note that for non-real characters $\chi$ modulo $T^{n-h}$, uniformly
	$$
	\bigg| \sum_{g \in \mathcal{M}_n} b(f) \chi(f)\bigg| = O_{n,h}(q^{n/2-1/2}),
	$$
	while from Lemma \ref{lem:realchar_bound} there are at most $O(q^{(n-h)/2}$ real non-trivial characters, and for such a character by Lemma \ref{lem:squareful_free_bound} this sum is $O_{n,h}(q^{n/2})$. Hence the left hand side of \eqref{eq:bound_to_show8.2} is at most
	$$
	\frac{q-1}{\Phi(T^{n-h})} \Big( \Phi_{ev}(T^{n-h})\cdot O_{n,h}(q^{n-1}) + O_{n,h}(q^n q^{(n-h)/2})\Big) = O_{n,h}(q^{n-1}).
	$$
\end{proof}

In a similar same manner, we obtain a more general bound for factorization functions that needn't be supported on the squarefuls.

\begin{lem}
	\label{lem:squareful_free_bound_short}
	For a fixed factorization function $a$,
	$$
	\sum_{f\in\mathcal{M}_n} \Big|\sum_{\substack{g\in I(f;h) \\f \in \mathcal{M}_n^\natural}} \widetilde{a}^\natural(g)\Big|^2 = O_{n,h}(q^{h+1} q^n),
	$$	
	for $0 \leq h \leq n$.
\end{lem}

\begin{proof}
	This follows from Lemma \ref{lem:short_to_prog} and Lemma \ref{lem:squareful_free_bound}. 
\end{proof}

\vspace{2mm}
\emph{\textbf{Step 3:}} We show that the variances of sums over $\mathcal{M}_n^\natural$ we have computed in Lemma \ref{lem:short_to_prog} are not far from those of sums over $\mathcal{M}_n$, which we are ultimately after.

\begin{lem}
	\label{lem:variance_approx}
	For a fixed factorization function $a$,
	$$
	\sum_{f\in\mathcal{M}_n} \Big| \sum_{g\in I(f;h)} \widetilde{a}(g)\Big|^2 = \sum_{f\in\mathcal{M}_n} \Big| \sum_{\substack{g\in I(f;h) \\ g \in \mathcal{M}_n^\natural}} \widetilde{a}^\natural(g)\Big|^2 + O_{h,n}(q^{h+1/2} q^n),
	$$
	for $0\leq h \leq n$.
\end{lem}
Note, in comparison with the error term, that we expect the left hand side to usually be of order $q^n q^{h+1}$.

\begin{proof}
	
	We make use of a mapping of polynomials $f\mapsto f^{[i]}$ defined by
	$$
	(a_0 + a_1 T + \cdots a_n T^n)^{[i]} = a_i + a_{i+1} T + \cdots + a_n T^{n-i},
	$$
	so that if $T^i | f$,
	$$
	f = T^i f^{[i]}.
	$$
	For $f\in \mathcal{M}_n$, we may partition $I(f;h)$ into the disjoint union
	\begin{equation}
		\label{eq:I_decompose}
		I(f;h) = \Big( \bigcup_{i=0}^h\{T^i g \in I(f;h):\; g \in \mathcal{M}_{n-i}^\natural\}\Big)\cup \{T^{h+1} f^{[h+1]}\}.
	\end{equation}
	For $g\in \mathcal{M}_{n-i}^\natural$, we define the function
	$$
	a_{[i]}(g) := a(T^i g),
	$$
	with
	$$
	\widetilde{a}_{[i]}(g) := a_{[i]}(g) - E^\natural(a_{[i]};n-i), \quad \mathrm{with}\quad E^\natural(a_{[i]};n-i)= \frac{1}{|\mathcal{M}_{n-i}^\natural|} \sum_{g\in \mathcal{M}_{n-i}^\natural} a_{[i]}(g).
	$$
	From the partitioning \eqref{eq:I_decompose}, it is easy to see that for $f\in\mathcal{M}_n$,
	\begin{align}
		\label{eq:a_decompose}
		\sum_{g\in I(f;h)} a(f) =& \sum_{\substack{g \in I(f;h) \\ g \in \mathcal{M}_n^\natural}}a(g) + \sum_{\substack{g \in I(f^{[1]};h-1) \\ g \in \mathcal{M}_{n-1}^\natural}} a_{[1]}(g)  \\
		\notag &+ \cdots + \sum_{\substack{g \in I(f^{[h]};0) \\ g \in \mathcal{M}_{n-h}^\natural}} a_{[h]}(g) + \underbrace{a(T^{h+1} f^{[h+1]})}_{=O_{n,h}(1)}.
	\end{align}
	Using that
	$$
	|\mathcal{M}_n^\natural| = q^{n-1}(q-1),
	$$
	and
	$$
	|\{g \in I(f;h):\; g\in \mathcal{M}_n^\natural\}| = q^h (q-1),
	$$
	one may verify (with a little work, but straightforwardly) that
	\begin{align}
		\label{eq:average_decompose}
		\notag \sum_{g\in I(f;h)} E(a;n) =& \frac{q^{h+1}}{q^n} \sum_{g\in \mathcal{M}_n} a(g) \\	
		=& \sum_{\substack{g\in I(f;h) \\ g \in \mathcal{M}_n^\natural}} E^\natural(a;n) + \sum_{\substack{g \in I(f^{[1]};h) \\ g \in \mathcal{M}_{n-1}^\natural}} E^\natural(a_{[1]};n-1)\\
		\notag &+ \cdots + \sum_{\substack{g \in I(f^{[h]};0) \\ g \in \mathcal{M}^\natural_{n-h}}} E^\natural(a_{[h]};n-h) + \underbrace{\frac{q^{h+1}}{q^n}\sum_{g\in \mathcal{M}_{n-h-1}} a(T^{h+1}g)}_{=O_{n,h}(1)}.		
	\end{align}
	
	Thus combining \eqref{eq:a_decompose} and \eqref{eq:average_decompose}, we have uniformly for $f \in \mathcal{M}_n$,
	\begin{align}
		\label{eq:shortsum_decompose}
		\sum_{g\in I(f;h)} \widetilde{a}(g) = &\sum_{\substack{g \in I(f;h) \\ g \in \mathcal{M}_n^\natural}} \widetilde{a}^\natural(g) + \sum_{\substack{g \in I(f^{[1]};h-1) \\ g \in \mathcal{M}_{n-1}^\natural}} \widetilde{a}_{[1]}^\natural(g) \\
		\notag &+ \cdots + \sum_{\substack{g \in I(f^{[h]};0) \\ g \in \mathcal{M}_{n-h}^\natural}} \widetilde{a}_{[h]}^\natural(g) + O_{n,h}(1).
	\end{align}
	For each $i$, the function $a_{[i]}$ defined on $\mathcal{M}_{n-i}$ extends uniquely to a factorization function defined on all of $\mathcal{M}_{n-i}$. Hence, using Lemma \ref{lem:squareful_free_bound_short} to pass to the second line below,
	\begin{align}
		\label{eq:i_bound}
		\notag \sum_{f\in\mathcal{M}_n}\Big|\sum_{\substack{g \in I(f^{[i]};h-i) \\ g \in \mathcal{M}_{n-i}^\natural}} \widetilde{a}^\natural_{[i]}(g)\Big|^2 =& q^i \sum_{f\in \mathcal{M}_{n-i}} \Big| \sum_{\substack{g \in I(f;h-i) \\ g\in \mathcal{M}_{n-i}^\natural}} \widetilde{a}^\natural_{[i]}(g)\Big|^2 \\
		&\ll_{n,h} q^i q^{h-i+1} q^{n-i}.
	\end{align}
	This quantity is no more than $q^h q^n$ for $i\geq 1$, and for $i=0$ it is of course equal to $q^{h+1}q^n$.
	
	Therefore, squaring the identity \eqref{eq:shortsum_decompose} and summing over $g\in\mathcal{M}_n$, then using Cauchy Schwarz and \eqref{eq:i_bound} to bound all terms but one on the right,
	$$
	\sum_{f\in\mathcal{M}_n} \Big|\sum_{g\in I(f;h)} \widetilde{a}(g)\Big|^2 = \sum_{f\in\mathcal{M}_n} \Big|\sum_{\substack{g\in I(f;h) \\ g \in \mathcal{M}_n^\natural}} \widetilde{a}^\natural(g)\Big|^2 +O_{n,h}( q^n q^{h+1/2}),
	$$
	as claimed.
\end{proof}

\vspace{2mm}
\emph{\textbf{Step 4:}}
Recall `factorization fourier expansion' \eqref{eq:factfourier}:
\begin{equation}
\label{eq:factfourier2}
a(f) = \underbrace{\sum_\lambda \hat{a}_\lambda X^\lambda(f)}_{=: A(f)} + b(f),
\end{equation}
where $b(f)$ is a function supported on the squarefuls. We use this to reduce variance for the function $a(f)$ to finding the covariance of characters $X^\lambda(f)$.

We introduce the shorthand, for partitions $\lambda, \nu \vdash n$,
\begin{equation}
	\label{eq:c_def}
	\Delta_{\lambda,\nu} (m):= \mathbb{E}_{\substack{\chi\; (T^m) \\ \textrm{prim., ev.}}} \Big(\sum_{f \in \mathcal{M}_n} X^\lambda(f) \chi(f)\Big)\overline{\Big(\sum_{g\in\mathcal{M}_n}X^\nu(g)\chi(g)\Big)}.
\end{equation}

Note that by Corollary \ref{cor:char_ortho}, for $m\geq 5$,
\begin{align}
\label{Del_eval}
\notag \Delta_{\lambda,\nu}(m) =& \delta_{\lambda\nu} \delta_{\ell(\lambda'),\ell(\nu') \leq m-2} + O(q^{-1/2}) \\
=& \delta_{\lambda\nu}\delta_{\lambda_1, \nu_1 \leq m-2} + O(q^{-1/2}).
\end{align}

\begin{lem}
	\label{lem:var_expand1}
	For a fixed factorization function $a$, with $0 \leq h \leq n-4$,
	\begin{equation}
		\label{eq:var_expand1}
		\Var_{f\in\mathcal{M}_n}\Big(\sum_{g\in I(f;h)} a(g)\Big) = q^{h+1} \sum_{\mu,\nu \vdash n} \Delta_{\lambda,\nu}(n-h) \hat{a}_\lambda \overline{\hat{a}_\nu} + O_{n,h}(q^{h+1/2}).
	\end{equation}
\end{lem}

\begin{proof}
	The variance in \eqref{eq:var_expand1} is given by
	$$
	\frac{1}{q^n}\sum_{f\in\mathcal{M}_n}\Big| \sum_{g\in I(f;h)} \widetilde{a}(g)\Big|^2 = \frac{1}{q^n} \sum_{f\in\mathcal{M}_n} \Big| \sum_{\substack{g \in I(f;h) \\ g \in \mathcal{M}_n^\natural}} \widetilde{A}^\natural(g)\Big|^2 + O_{n,h}(q^{h+1/2}),
	$$
	where we have reduced to a sum of terms $\widetilde{A}^\natural(g)$ by using Lemmas \ref{lem:variance_approx} and then \ref{lem:squareful_bound_short} and \ref{lem:squareful_free_bound_short} to absorb a sum of terms $b^\natural(g)$ into the error term.
	
	In turn from Lemma \ref{lem:short_to_prog},
	\begin{align*}
		\frac{1}{q^n} \sum_{f\in\mathcal{M}_n} \Big|\sum_{g\in I(f;h)} \widetilde{A}^\natural(g)\Big|^2 &= \frac{q^{h+1}(q-1)}{q^n \Phi(T^{n-h})} \sum_{\substack{\chi \neq \chi_0\;(T^{n-h}) \\ \textrm{even}}} \Big| \sum_{g\in\mathcal{M}_n} A(g)\chi(g)\Big|^2 \\
		&= \frac{q^{h+2}}{q^n q^{n-h}} \Big( \sum_{\substack{\chi \;(T^{n-h}) \\ \textrm{prim., ev.}}} \Big|\sum_{g\in\mathcal{M}_n} A(g)\chi(g)\Big|^2	+ O_{n,h}(q^n \cdot q^{n-h-2})\Big).
	\end{align*}
	The second line has followed by taking non-primitive even characters from the sum and bounding their contribution by Lemma \ref{lem:squareful_free_bound}.
	The above quantity simplifies to
	$$
	q^{h+1} \mathbb{E}_{\substack{\chi \;(T^{n-h}) \\ \textrm{prim., ev.}}} \Big| \sum_{g\in\mathcal{M}_n} A(g)\chi(g)\Big|^2 + O_{n,h}(q^h),
	$$
	and Lemma \ref{lem:var_expand1} follows by expanding $A(g)$ into a linear combination of characters $X^\lambda$ (recall $A$ is defined by \eqref{eq:factfourier2}) and then expanding the square above.
\end{proof}

With this lemma in place, Theorem \ref{thm:general_var} now follows by applying \eqref{Del_eval}.





\section{Factorization Fourier expansions}
\label{sec:fact_exp}

\subsection{}
We list some examples of the expansion \eqref{eq:factfourier} for the arithmetic functions we considered in section \ref{sec:1}. In this way we recover Theorems \ref{thm:a_mobiusvar}, \ref{thm:b_mangoldtvar}, and \ref{thm:f_dvar}, estimating variance over short intervals of the M\"obius function, the von Mangoldt function, and the $k$-fold divisor function. We also consider the function $\omega$, which as usual counts distinct prime factors, and this leads to a new result for the variance of $\omega(f)$ and $\mu(f)\omega(f)$ summed over short intervals.


\begin{prop}
	\label{prop:mobius_exp}
	For $f\in \mathcal{M}_n$,
	$$
	\mu(f) = (-1)^n X^{(1^n)}(f).
	$$
\end{prop}

\textbf{Remark:} Applied to Theorem \ref{thm:general_var} this recovers Theorem \ref{thm:a_mobiusvar}, for $\mu(f)$.

\begin{proof}
	Both $\mu(f)$ and $X^{(1^n)}(f)$ will be zero unless $f$ is squarefree. But for $f = p_1\cdots p_\ell$, with all factors distinct, $\mu(f) = (-1)^\ell$, while it may be checked $X^{(1^n)}(f) = (-1)^{\deg(p_1)-1}\cdots (-1)^{\deg(p_\ell)-1}.$ As $(-1)^{\deg(p_1)}\cdots (-1)^{\deg(p_\ell)} = (-1)^n$, this verifies the claim.
\end{proof}

\begin{prop}
	\label{prop:mobius2_exp}
	For $f\in \mathcal{M}_n$,
	$$
	\mu(f)^2 = X^{(n)}(f).
	$$
\end{prop}

\begin{proof}
	As $X^{(n)}$ is the trivial character, this is clear.
\end{proof}

\begin{prop}
	\label{prop:lambda_exp}
	For $f\in \mathcal{M}_n$,
	$$
	\Lambda(f) = \sum_{r=1}^{n} (-1)^{n-r} X^{(r,1^{n-r})}(f) + b(f),
	$$
	for a function $b(f)$ that is supported on the squarefuls.
\end{prop}

\textbf{Remark:} This recovers Theorem \ref{thm:b_mangoldtvar}, for $\Lambda(f)$.

This is a special case of:

\begin{prop}
	\label{prop:lambdaj_exp}
	For $f\in \mathcal{M}_n$,
	$$
	\Lambda_j(f) = \sum_{r=1}^{n} (-1)^{n-r} (r^j - (r-1)^j) X^{(r,1^{n-r})}(f) + b(f),
	$$
	for a function $b(f)$ that is supported on the squarefuls.
\end{prop}

\textbf{Remark:} This recovers an estimate for the covariance of almost-primes in short intervals, proved in \cite{Ro}.

Note that
$$
\Lambda_j(f) := \sum_{\substack{g|f \\ g\;\textrm{monic}}} \mu(g) \deg(f/g)^j = \sum_{r=1}^n \big(r^j-(r-1)^j\big) \sum_{\substack{g|f \\ \deg(g) \leq n-r \\ g\; \textrm{monic}}} \mu(g),
$$
so we have that Proposition \ref{prop:lambdaj_exp} is a corollary of

\begin{prop}
	\label{prop:mupartial_exp}
	For $f\in \mathcal{M}_n$, with $n = r+s$,
	$$
	\sum_{\substack{g|f \\ \deg(g) \leq s \\ g\; \textrm{monic}}} \mu(f) = (-1)^s X^{(r,1^s)}(f) + b(f),
	$$
	for a function $b(f)$ that is supported on the squarefuls.
\end{prop}

\begin{proof}
	We will need to make use of the Murnaghan-Nakayama rule, quoted in Theorem \ref{thm:murnaghan}. 
	
	We may suppose that $f$ is squarefree (otherwise the proposition is trivial), and let $f = p_1\cdots p_\ell$ with $\deg p_i = \tau_i$, $\tau_1\geq \tau_2 \geq ...$. We apply the Murnaghan-Nakayama rule to the type $\tau_f = (\tau_1,...,\tau_\ell)$ and Young diagram of $(r,1^s)$. For any border-strip tableau, let $I \subset \{2,...,\ell\}$ be the collection of numbers that appear in rows $2$ through $s$ of the Young diagram of $(r,1^s)$. Writing
	$$
	\tau_I := \sum_{i\in I} \tau_i,
	$$
	to form a valid border-strip tableau, it is easy to see that we require only that $\tau_I \leq s$ and $\tau_1 + \tau_I \geq s+1$. Hence, applying the rule,
	\begin{align}
		\label{eq:mn_hook}
		\notag X^{(r,1^s)}(f) &= \sum_{\substack{I \subset \{2,...,\ell\} \\ \tau_I \leq s \\ \tau_1 + \tau_I \geq s+1}} (-1)^{\tau_I - |I|}(-1)^{(s+1)-\tau_I -1} \\
		&= (-1)^s \sum_{\substack{I \subset \{2,...,\ell\} \\ s-\tau_1 < \tau_I \leq k}} (-1)^{|I|}.
	\end{align}
	Yet,
	\begin{equation}
		\label{eq:mobius_trunc}
		\sum_{\substack{g|f \\ \deg(g) \leq s \\ g \;\textrm{monic}}} \mu(g) = \sum_{\substack{J \subset \{1,...,\ell\}}} (-1)^{|J|}
	\end{equation}
	By breaking the right-hand sum into parts for which $1$ is an element of $J$ or not, we see that \eqref{eq:mobius_trunc} is equal to
	$$
	\sum_{\substack{I \subset \{2,...,\ell\} \\ \tau_I \leq s}} (-1)^{|J|} + \sum_{\substack{I \subset \{2,...,\ell\} \\ \tau_I + \tau_1 \leq s}} (-1)^{|J|+1} = \sum_{\substack{ I \subset \{2,...,\ell\} \\ s-\tau_1 < \tau_I \leq s}} (-1)^{|J|}.
	$$
	Comparing this with \eqref{eq:mn_hook} yields the result.
\end{proof}

\begin{prop}
	\label{prop:divisor_exp}
	For $f\in\mathcal{M}_n$,
	\begin{equation}
	\label{eq:divisor_exp1}
	d_k(f) = \sum_{\substack{\lambda \vdash n \\ \ell(\lambda) \leq k}} s_\lambda(\underbrace{1,...,1}_k) X^\lambda(f) + b(f),
	\end{equation}	
	for a function $b(f)$ that is supported on the squarefuls. Moreover, we have the following equivalent expressions for $s_\lambda(1,...,1)$: 
	\begin{enumerate}
	\item
	\begin{equation}
	\label{eq:divisor_exp2}
	s_\lambda(\underbrace{1,...,1}_k) = \prod_{1\leq i < j \leq k} \frac{\lambda_i - \lambda_j + j-i}{j-i},
	\end{equation}
	with the convention if $\ell(\lambda) < k$ that $\lambda_{\ell(k)+1} = \cdots = \lambda_k = 0.$
	
	\item
	\begin{equation}
		\label{eq:divisor_exp3}
		s_\lambda(\underbrace{1,...,1}_k) = GT_k(\lambda)
	\end{equation}
	where $GT_k(\lambda)$ is the number of triangular arrays of non-negative integers
	$$
	\begin{array}{cccccccccccccc} x^{(1)}_1 & & x^{(1)}_2 & & \cdots &  & x^{(1)}_k \\ & \ddots & & \ddots &  & &   & & \\    &  &   x^{(k-1)}_1 &  & x^{(k-1)}_2 \\  \\ &  &    & x^{(k)}_1 &\end{array}
	$$
	with entries weakly decreasing left-to-right down diagonals and weakly increasing left-to-right up diagonals (that is, $x^{(i)}_j \geq x^{(i+1)}_j \geq x^{(i)}_{j+1}$), and in the top row, $x^{(1)}_i = \lambda_i$, with again the convention if $\ell(\lambda) < k$ that $\lambda_{\ell(k)+1} = \cdots = \lambda_k = 0.$

	\item
	\begin{equation}
		\label{eq:divisor_exp4}
		s_\lambda(\underbrace{1,...,1}_k) = \prod_{u\in\lambda} \frac{k + c(u)}{h(u)},
	\end{equation}
	where the product is over all squares $u$ of the Young diagram of $\lambda$, and where if we label the squares $u$ by the coordinates $(i,j)$ with $1 \leq j \leq \lambda_i$, the \emph{content} $c(u)$ is defined by
	$$
	c(u) = i-j,
	$$
	and the \emph{hook length} $h(u)$ is defined by
	$$
	h(u) = \lambda_i + \lambda_j' - i - j +1.
	$$
	(See \cite[p. 373]{St} for a lengthier account of these definitions.)
	
	\end{enumerate} 	
\end{prop}

\textbf{Remark:} Using the representation \emph{(ii)}, this recovers the variance of the $k$-fold divisor function given in Theorem \ref{thm:f_dvar}.

\begin{proof}
	It will again be sufficient to consider $f$ squarefree. We note that for $p$ prime, $d_k(p) = k$, so for $f = p_1\cdots p_\ell$ with all prime factors distinct,
	$$
	d_k(f) = k^\ell = k^{\ell(\tau)},
	$$
	where $\tau$ is the factorization type of $f$. On the other hand,
	\begin{align}
		\label{eq:powerlength}
		\notag k^{\ell(\tau)} &= p_{\tau}(\underbrace{1,...,1}_k) \\
		&= \sum_{\lambda \vdash n} s_\lambda(\underbrace{1,...,1}_k) X^\lambda(\tau),
	\end{align}
	by Theorem \ref{thm:frobenius} of Frobenius. This proves \eqref{eq:divisor_exp1}.
	
	For the formula given in \emph{(i)}, note that for $\ell(\lambda) > k$, we have $s_\lambda(1,...,1) = 0$, while for $\ell(\lambda) \leq k$, the identity \eqref{eq:divisor_exp2} is \cite[Ex. 6, Ch. 6]{Fu}.
	
	For the formula given in \emph{(ii)}, note that $s_\lambda(1,...,1)$ is equal to the number of semistandard Young tableaux of shape $\lambda$ with entries $1$ through $k$ (see \cite[Sec. 7.10]{St}), and by a well-known bijection (again, see \cite[Sec. 7.10]{St}) this is equal to $GT_k(\lambda)$. (For readers familiar with the terminology, $GT_k(\lambda)$ is a count of Gelfand-Tsetlin patterns.)
	
	For the formula given in \emph{(iii)}, this is Corollary 7.21.4 of \cite{St}.
\end{proof}

\begin{prop}
	\label{prop:omega_exp}
	Let $\omega(f)$ be the number of distinct primes that divide $f$. Then for $f\in\mathcal{M}_n$,
	\begin{equation}
		\label{eq:omega_exp1}
		\omega(f) = H_n X^{(n)}(f) + \sum_\lambda (-1)^\nu \Big(\frac{1}{\lambda_2+\nu}-\frac{1}{\lambda_1+\nu+1}\Big) X^{(\lambda_1,\lambda_2,1^\nu)}(f) + b(f),
	\end{equation}
	where the sum is over all partitions $\lambda = (\lambda_2,\lambda_1,1^\nu) \vdash n$ with $\lambda_2 \geq 1$ and $\nu \geq 0$, where $b(f)$ is a function supported on the squarefuls, and where
	$$
	H_n:= \frac{1}{1} + \frac{1}{2} + \cdots \frac{1}{n}.
	$$
\end{prop}

\textbf{Remark:} The mean value as $q\rightarrow\infty$ of $\omega(f)$ for $\deg(f) = n$ is $H_n$. Because $X^{(n)}(f) = 1$ for all squarefree $f$, the expression \eqref{eq:omega_exp1} may be thought of as characterizing the oscillation of $\omega(f)$ around this value.

\begin{proof}
We use the identity \eqref{eq:powerlength} from the last proof, along with the representation \eqref{eq:divisor_exp4} for $s_\lambda(1,...,1)$. Taken together these imply for $\tau \vdash n$ and positive integer $k$,
\begin{equation}
	\label{eq:power_exp}
	k^{\ell(\tau)} = \sum_{\lambda \vdash n} \prod_{u \in \lambda} \frac{k+c(u)}{h(u)} X^\lambda(\tau).
\end{equation}
Though we have only demonstrated \eqref{eq:power_exp} for integer $k$, both the left and right hand side of this identity are polynomials in $k$, and therefore \eqref{eq:power_exp} must hold for all $k\in \mathbb{C}$. Differentiating \eqref{eq:power_exp} and setting $k=1$ requires some slightly tedious book-keeping, but is otherwise straightforward and gives us
\begin{equation}
	\label{eq:ell_exp}
	\ell(\tau)	 = H_n X^{(n)}(\tau) + \sum_\lambda (-1)^\nu \Big(\frac{1}{\lambda_2+\nu}-\frac{1}{\lambda_1+\nu+1}\Big) X^{(\lambda_1,\lambda_2,1^\nu)}(\tau).
\end{equation}
Applying this to the factorization types of $f \in \mathcal{M}_n$ gives the proposition.
\end{proof}

\begin{prop}
	\label{prop:muomega_exp}
	For $f \in \mathcal{M}_n$,
	\begin{align}
		\label{eq:muomega_exp1}
		\mu(f)\omega(f) = (-1&)^n\Big[ H_n X^{(1^n)}(f)  \\&+ \sum (-1)^\nu \Big(\frac{1}{j+\nu+1} - \frac{1}{i+j+\nu+2}\Big) X^{(\nu+2, 2^j, 1^i)}(f)\Big] + b(f),\notag
	\end{align}
	where the sum is over all partitions $(\nu +2, 2^j, 1^i) \vdash n$, with $i,j,\nu \geq 0$, and $b(f)$ is a function supported on the squarefuls.
\end{prop}

\begin{proof}
	For $f$ squarefree with factorization type $\tau$, note that $\mu(f)\omega(f) = (-1)^{\ell(\tau)} \ell(\tau)$. But by applying Proposition \ref{prop:dual_part} to the identity \eqref{eq:ell_exp}, we may decompose $(-1)^{\ell(\tau)}\ell(\tau)$ into a sum over irreducible characters associated to dual partitions. This decomposition yields \eqref{eq:muomega_exp1}.
\end{proof}

\subsection{}
By applying Theorem \ref{thm:general_var} to Propositions \ref{prop:omega_exp} and \ref{prop:muomega_exp} we straightforwardly obtain the following results.

\begin{cor}
	\label{cor:omega_var}
	For fixed $0 \leq h \leq n-5$,
	\begin{multline}
		\label{eq:omega_var}
		\Var_{f\in\mathcal{M}_n} \Big(\sum_{g\in I(f;h)} \omega(g) \Big) \\= q^{h+1} \mathop{\sum\sum}_{\substack{1 \leq \lambda_1 \leq \lambda_2 \leq n-h-2 \\ \lambda_1+\lambda_2 \leq n}} \Big(\frac{1}{n-\lambda_1}\, -\, \frac{1}{n-\lambda_2+1}\Big)^2 + O(q^{h+1/2}).
	\end{multline}
\end{cor}

\begin{cor}
	\label{cor:muomega_var}
	For fixed $0 \leq h \leq n-5$,
	\begin{multline}
		\label{eq:muomega_var}
		\Var_{f\in\mathcal{M}_n} \Big(\sum_{g \in I(f;h)} \mu(g)\omega(g)\Big) \\= q^{h+1} \Big[ H_n^2 + \mathop{\sum\sum}_{h+2 \leq i + 2j \leq n-2} \Big(\frac{1}{n-i-j-1}\,-\, \frac{1}{n-j}\Big)^2 \Big]+O(q^{h+1/2}).
	\end{multline}
\end{cor}

\subsection{}
Because the double-indexed sum in the asymptotic formula of \eqref{eq:muomega_var} remains bounded for $n\rightarrow\infty$ and $h \sim \delta n$ with $\delta > 0$, and because $H_n \sim \log n = \log \deg(f)$ for $f \in \mathcal{M}_n$, one may think of Corollary \ref{cor:muomega_var} as a function field analogue of the following conjecture over the integers (which is intuitive enough on its own):

\begin{conj}
	\label{conj:muomega_var} For $H = X^\delta$ with fixed $\delta \in (0,1)$, as $X\rightarrow\infty$, we have
	$$
	\frac{1}{X}\int_X^{2X} \Big(\sum_{x \leq n \leq n+H} \mu(n)\omega(n) \Big)^2\, dx \sim H (\log \log X)^2.
	$$
\end{conj}

Corollary \ref{cor:omega_var} has a rather more striking interpretation. In \eqref{eq:omega_var} the double indexed sum remains bounded for $h \sim \delta n$ with $\delta \in (0,1)$ fixed; indeed the reader may check that
\begin{equation}
\label{eq:doublesum_asymp}
\mathop{\sum\sum}_{\substack{1 \leq \lambda_1 \leq \lambda_2 \leq n-h-2 \\ \lambda_1+\lambda_2 \leq n}} \Big(\frac{1}{n-\lambda_1}\, -\, \frac{1}{n-\lambda_2+1}\Big)^2 \sim p(\delta) < +\infty
\end{equation} 
as $n\rightarrow\infty$ for \footnote{One can further reduce the integral to see
$$
p(\delta) = \begin{cases} \log\Big(\frac{1-\delta}{\delta}\Big) + \delta - \mathrm{Li}_2(1-\delta) + \mathrm{Li}_2(\delta) - \log(1-\delta)\log(\delta) & \textrm{for}\; \delta \leq 1/2 \\
\frac{1-\delta}{\delta} - (1-\delta) - \log(\delta)^2 & \textrm{for}\; \delta > 1/2,
\end{cases}
$$
where $\mathrm{Li}_2$ is the dilogarithm. Note the phase change at $\delta = 1/2$.
}
$$
p(\delta) := \int_{\substack{x+y \geq 1 \\ \delta \leq x \leq y \leq 1}} \Big(\frac{1}{x} - \frac{1}{y}\Big)^2 \, dx dy 
$$
Because this is bounded it is reasonable to suppose

\begin{conj}
\label{conj:omega_var}
For $H = X^\delta$ with fixed $\delta \in (0,1)$ as $X\rightarrow\infty$ we have
\begin{equation}
	\label{eq:omega_var}
	\frac{1}{X}\int_X^{2X} \Big(\sum_{x \leq n \leq x+H} \omega(n)\Big)^2\, dx - \Big( \frac{1}{X} \int_X^{2X} \sum_{x \leq n \leq x+H} \omega(n)\, dx\Big)^2 = O_\delta(H).
\end{equation}
\end{conj}

There is a sense in which an estimate of the sort \eqref{conj:omega_var} would be surprising, since the Erd\H{o}s-Kac theorem \cite{ErKa} predicts that diagonal terms make a contribution of size $H\log \log X$. Clearly that $\delta \in (0,1)$ remain fixed is important for anything like \eqref{conj:omega_var} to be true -- the consideration of diagonal terms shows that we cannot have such an estimate if $\delta \rightarrow 0$ as $X \rightarrow \infty$. Nonetheless the function field analogy remains, and it would be interesting to study in greater depth whether Conjecture \ref{conj:omega_var} is true.\footnote{Andrew Granville (personal communication) has shown a variant of this conjecture is true for a restricted range of $\delta$, when $\omega(n)$ is replaced by $\omega_y(n)$, a count of prime factors of $n$ less than $y=X^{1/2-\epsilon}$.}

Rather more ambitiously, one may even guess that the right hand side of \eqref{eq:omega_var} can be replaced by
$$
p(\delta) H + o_\delta(H).
$$
 
\section{Covariance}
\label{sec:covariance}

\subsection{}
In analogy with the definition \eqref{def:variance} of variance, we define the covariance of two arithmetic functions $\eta_1$ and $\eta_2$ by
$$
	\Covar_{f\in\mathcal{M}_n}\big(\eta_1(f)\,,\, \eta_2(f)\big) := \frac{1}{q^n} \sum_{f\in \mathcal{M}_n} (\eta_1(f) - \mathbb{E}_{\mathcal{M}_n}\eta_1) \overline{(\eta_2(f) - \mathbb{E}_{\mathcal{M}_n} \eta_2)}.
$$

Because Theorem \ref{thm:general_var} holds for general factorization function $a$, it implies by a standard argument a corresponding result for covariance.
\begin{thm}
	\label{thm:general_covar}
	For $a(f)$ and $b(f)$ fixed factorization functions, and for fixed $0 \leq h \leq n-5$,
	$$
	\Covar_{f\in \mathcal{M}_n}\Big(\sum_{g\in I(f;h)} a(g)\,,\, \sum_{g\in I(f;h)} b(g)\Big) = q^{h+1} \sum_{\substack{\lambda \vdash n \\ \lambda_1 \leq n-h-2}} \hat{a}_\lambda \overline{\hat{b}_\lambda} + O(q^{h+1/2}).
	$$
\end{thm}

One consequence of this is worthwhile to draw out. Since $\mu(g) = X^{(1^n)}(g),$ we see directly that
\begin{cor}
	\label{cor:mu_covar}
	For $a(f)$ a fixed factorization function and for fixed $0 \leq h \leq n-5$,
	\begin{align*}
	\Covar_{f\in \mathcal{M}_n}\Big(\sum_{g\in I(f;h)} a(g)\,,\, \sum_{g\in I(f;h)} \mu(g)\Big) &= q^{h+1} \hat{a}_{(1^n)} + O(q^{h+1/2}) \\
	&= q^{h+1} \cdot \frac{1}{q^n}\sum_{f\in\mathcal{M}_n} \mu(g) a(g) + o(q^{h+1}).
	\end{align*}
\end{cor}
That is to say, the M\"obius function oscillates to such an extent that in estimating its short-interval-sum covariance against any factorization function, only diagonal terms contribute. Is is easy to see that (up to values on the squarefuls) the M\"obius function is unique among factorization functions in this regard.

For example, Corollary \ref{cor:mu_covar} implies
$$
\Covar_{f\in \mathcal{M}_n}\Big(\sum_{g\in I(f;h)} \Lambda(g)\,,\, \sum_{g\in I(f;h)} \mu(g)\Big) \sim - q^{h+1},
$$
as $q\rightarrow\infty$. Over the integers we have the following analogy:
\begin{conj}
	\label{conj:covar_Lambdamu}
	For $H = X^\delta$ with $\delta \in (0,1)$,
	$$
	\frac{1}{X} \int_X^{2X} \Big(\sum_{x \leq n \leq x+H} \Lambda(n)-H\Big) \Big(\sum_{x \leq n \leq n+H} \mu(n)\Big) \, dx \sim \,-H,
	$$
	as $X\rightarrow\infty$.
\end{conj} 

\section{Decompositions: proofs of Theorem \ref{thm:subspace_decompose1} and Corollary \ref{cor:variance_to_minimizer}}
\label{sec:10.1}

\subsection{} We now turn to the decomposition of the space of factorization functions $\mathcal{F}$ into $\mathcal{U}_n^h$ and $\mathcal{V}_n^h$ and the corresponding evaluation of variance described in Theorem \ref{thm:subspace_decompose1}. Recall that $\mathcal{U}_n^h$ is the linear space of functions defined by \eqref{eq:udef}, and $\mathcal{V}_n^h$ is orthogonal complement supported on squarefuls. We first demonstrate the explicit characterization of the spaces $\mathcal{U}_n^h$ and $\mathcal{V}_n^h$ given by Proposition \ref{prop:subspace_decompose2}.

\begin{proof}[Proof of Proposition \ref{prop:subspace_decompose2}]
Let $\mathcal{A}_n^h$ and $\mathcal{B}_n^h$ be as in the Proposition and
$$
\mathcal{C}_n^h := \mathrm{span}\{ X^\lambda(f)\; : \; \lambda \vdash n,\, \lambda_1 \leq n -h -2\}.
$$

Note that $\mathcal{C}_n^h$ is supported on the squarefrees, and
$$
\mathcal{F} = (\mathcal{A}_n^h\oplus \mathcal{B}_n) \oplus \mathcal{C}_n^h.
$$
Moreover, by the equidistribution of factorization types and cycles types and the orthogonality of characters $X^\lambda$, $\mathcal{A}_n^h$ is orthogonal to $\mathcal{C}_n^h$.

Theorem \ref{thm:general_var} implies that $\mathcal{A}_n^h\oplus \mathcal{B}_n \subset \mathcal{U}_n^h$, and likewise that $\mathcal{C}_n^h \cap \mathcal{U}_n^h = \{0\}$, so that no function outside of $\mathcal{A}_n^h\oplus \mathcal{B}_n$ lies in $\mathcal{U}_n^h$; that is, $\mathcal{A}_n^h\oplus \mathcal{B}_n = \mathcal{U}_n^h$. $\mathcal{V}_n^h$, defined to be the orthogonal complement supported on squarefuls, is thus identical with $\mathcal{C}_n^h$, which proves the proposition.
\end{proof}

\begin{proof}[Proofs of Theorem \ref{thm:subspace_decompose1}]
Note that for $v\in\mathcal{V}_n^h$ with
$$
v(f) = \sum_{\lambda_1 \leq n-h-2} \hat{v}_\lambda X^\lambda(f),
$$
we have
$$
\langle v, v \rangle = \lim_{q\rightarrow\infty} \frac{1}{q^n} \sum_{f\in\mathcal{M}_n} v(f)\overline{v(f)} = \sum_{\lambda_1 \leq n-h-2} |\hat{v}_\lambda|^2,
$$
by again making use of the equidistribution of factorization types and cycle types (Proposition \ref{prop:factor_to_cycle}). Combined with Theorem \ref{thm:general_var}, this gives the result.
\end{proof}

\subsection{}
We now turn to Proposition \ref{prop:subspace_decomposedivisors} and Corollary \ref{cor:variance_to_minimizer}.

\begin{proof}[Proof of Proposition \ref{prop:subspace_decomposedivisors}]
We note first that for any factorization function $\alpha$, it is simple to see that
$$
w(f):= \sum_{\substack{\delta | f \\ \deg(\delta) \leq h+1}} \alpha(\delta),\quad (\textrm{defined for}\; f\in \mathcal{M}_n)
$$
lies in $\mathcal{U}_n^h$. (Recall that $\mathcal{U}_n^h$ is defined by \eqref{eq:udef}.) For in this case, for any $f\in \mathcal{M}_n$,
$$
\sum_{g \in I(f;h)} q(g) = \sum_{\deg(\delta) \leq h+1} \alpha(\delta) \sum_{\substack{g \in I(f;h) \\ \delta | g}} 1 = \sum_{\deg(\delta) \leq h+1} \alpha(\delta) q^{h+1-\deg(\delta)}.
$$
This does not depend on $f$, so that
$$
\Var_{f\in \mathcal{M}_n}\Big(\sum_{g\in I(f;h)} w(g)\Big) = 0,
$$
implying $w\in \mathcal{U}_n^h$. Since we already know any factorization function $b \in \mathcal{F}_n$ supported on the squarefuls lies in the linear space $\mathcal{U}_n^h$, and function of the form $w(f)+b(f)$ must therefore lie in $\mathcal{U}_n^h$.

Hence to complete the proof of the proposition, we need only show that all functions in $\mathcal{U}_n^h$ are of this form. Having already characterized $\mathcal{U}_n^h$ in terms of characters of the symmetric group in Proposition \ref{prop:subspace_decompose2}, we will have done so if we show that for $\lambda \vdash n$ with $\lambda_1 \geq n - h -1$, there exists a factorization function $\alpha$ and a factorization function $b$ supported on the squarefuls such that
$$
X^\lambda(f) = \sum_{\substack{\delta|f \\ \deg(\delta) \leq h+1}} \alpha(\delta) + b(f), \quad (\textrm{for all}\;f\in \mathcal{M}_n).
$$
The remainder of this proof is devoted to a demonstration in four steps of this claim.

\textit{\textbf{Step 1:}} Let $m$ be arbitrary. For an even primitive character $\chi$ modulo $T^m$, from the identity
$$
\Big(1-\frac{u}{\sqrt{q}}\Big) \prod_{j=1}^{m-2}(1-u e^{i 2\pi \vartheta_j}) = \mathcal{L}\Big(\frac{u}{\sqrt{q}},\chi\Big) = \sum_{n\geq 0} u^n \frac{1}{q^{n/2}}\sum_{f \in \mathcal{M}_n} \chi(f),
$$
we have the following expression for elementary symmetric functions in the normalized roots of the $\mathcal{L}$-function:
\begin{equation}
\label{eq:elsym_to_char}
e_n(\Theta_\chi) = \frac{(-1)^n}{q^{n/2}} \sum_{f\in \mathcal{M}_n} \chi(f) + O_{n,m}(q^{-1/2}).
\end{equation}

\textit{\textbf{Step 2:}} We note for $n_1 + \cdots + n_k = n$,
\begin{multline*}
e_{n_1}(\Theta_\chi) \cdots e_{n_k}(\Theta_\chi) \\
= \frac{(-1)^n}{q^{n/2}} \Big( \sum_{f_1 \in \mathcal{M}_{n_1}} \chi(f_1) + O_{n,m}(q_{-1/2})\Big) \cdots \Big( \sum_{f_1k \in \mathcal{M}_{n_k}} \chi(f_k) + O_{n,m}(q_{-1/2})\Big) \\
\\= \frac{(-1)^n}{q^{n/2}} \sum_{\substack{f_1 \in \mathcal{M}_{n_1} \\ g \in \mathcal{M}_{n_2+\cdots+n_k}}} \chi(f_1 g) \alpha(g) + O_{n,m}(q^{-1/2}),
\end{multline*}
where
$$
\alpha(g):= \sum_{\substack{f_2\cdots f_k = g \\ f_2 \in \mathcal{M}_2, ..., f_k \in \mathcal{M}_k}} 1
$$
is a factorization function supported on $\mathcal{M}_{n_2+\cdots+ n_k}$. In particular, we have that if $n_1 \geq n - h -1$ (so $n_2+\cdots + n_k \leq h+1$) then
\begin{equation}
\label{eq:prodelsym_to_char}
e_{n_1}\cdots e_{n_k} = \frac{(-1)^n}{q^{n/2}} \sum_{f\in \mathcal{M}_n} \chi(f) \sum_{\delta |f} \alpha(\delta) + O_{n,m}(q^{-1/2}),
\end{equation}
for a factorization function $\alpha(\delta)$ supported on the set of $\delta$ with $\deg(\delta) \leq h+1$.

\textit{\textbf{Step 3:}} From an expansion of the determinant in the Jacobi-Trudi identity, we see for $\lambda \vdash n$ that $s_{\lambda'}$ is a linear combination of terms $e_{n_1}\cdots e_{n_k}$ with $n_1+ \cdots n_k = n$ and (from the top row of the determinant) $n_1 \geq \lambda_1$ always. Hence via step 2, if $\lambda_1 \geq n-h-1$,
\begin{equation}
\label{eq:schur_to_char}
s_{\lambda'}(\Theta_\chi) = \frac{(-1)^n}{q^{n/2}} \sum_{f\in \mathcal{M}_n} \chi(f) \sum_{\delta|f} \alpha(\delta) + O_{n,m}(q^{-1/2}),
\end{equation}
for a factorization function $\alpha(\delta)$ supported on $\delta$ with $\deg(\delta) \leq h+1$, since linear combinations of terms of the form $\sum_{\delta|f}\alpha(\delta)$ remain of this form.

Yet from Theorem \ref{thm:schur_in_zeros}
\begin{equation}
\label{eq:schur_to_char2}
s_{\lambda'}(\Theta_\chi) = \frac{(-1)^n}{q^{n/2}}\sum_{f\in \mathcal{M}_n} X^\lambda(f) \chi(f) + O(q^{-1/2}).
\end{equation}
Hence pairing \eqref{eq:schur_to_char} and \eqref{eq:schur_to_char2} we have
\begin{equation}
\label{eq:schur_to_char3}
\frac{1}{q^{n/2}} \sum_{f\in \mathcal{M}_n} \chi(f) \Big(X^\lambda(f) - \sum_{ \substack{\delta|f \\ \deg(\delta)\leq h+1}} \alpha(\delta)\Big) = O_{n,m}(q^{-1/2}).
\end{equation}

\textit{\textbf{Step 4:}} In \eqref{eq:schur_to_char3}, $m$ is arbitrary; take $m$ sufficiently large depending on $n$, with the intention of using Lemma \ref{lem:diag_by_char}. We have upon squaring and averaging,
$$
\mathbb{E}_{\substack{\chi\, (T^m) \\ \mathrm{prim., ev.}}} \Big| \frac{1}{q^{n/2}} \sum_{f\in \mathcal{M}_n} \chi(f) \Big(X^\lambda(f) - \sum_{ \substack{\delta|f \\ \deg(\delta)\leq h+1}} \alpha(\delta)\Big) \Big|^2 \rightarrow 0,
$$
as $q\rightarrow \infty$. But then from Lemma \ref{lem:diag_by_char},
$$
\Big\| X^\lambda(f) - \sum_{ \substack{\delta|f \\ \deg(\delta)\leq h+1}} \alpha(\delta) \Big\| = 0,
$$
for $\|\cdot\|$ the norm induced by our inner product. Since this inner product is non-degenerate on functions supported on the squarefrees, we must have
$$
X^\lambda(f) = \sum_{ \substack{\delta|f \\ \deg(\delta)\leq h+1}} \alpha(\delta) + b(f),
$$
for some function $b(f)$ supported on the squarefuls, as claimed.
\end{proof}

\begin{proof}[Proof of Corollary \ref{cor:variance_to_minimizer}]
This follows immediately from Theorem \ref{thm:subspace_decompose1} and Proposition \ref{prop:subspace_decomposedivisors}. For in the identity \ref{eq:var_project}, the function $v(f)$ is a projection of the function $a(f)$ to the subspace $\mathcal{V}_n^h$, but then
$$
\langle v, v \rangle = \| \mathrm{Proj}_{\mathcal{V}_n^h}( a )\|^2 = \inf_{u \in \mathcal{U}_n^h} \| a - u \|^2 = \inf_{\alpha \in \mathcal{F}} \Big\| a(f)  - \sum_{ \substack{\delta|f \\ \deg(\delta)\leq h+1}} \alpha(\delta)\Big\|^2.
$$
\end{proof}

\subsection{}
It is worthwhile to reflect one last time on the dichotomy between $\mathcal{U}_n^h$ and $\mathcal{V}_n^h$. Theorem \ref{thm:schur_in_zeros} gives us another way to characterize them. $\mathcal{U}_n^h$ is just the collection of those factorization functions $u$ for which 
\begin{equation}
\label{eq:trivial_character_bound}
\sum_{f\in \mathcal{M}_n} u(f)\chi(f) = O(q^{n/2-1/2}),
\end{equation}
uniformly for all even primitive characters modulo $T^{n-h}$. The reason that Theorem \ref{thm:schur_in_zeros} implies \eqref{eq:trivial_character_bound} is very simply that $\mathcal{L}(u,\chi)$ has always $n-h-2$ non-trivial zeros. Contrariwise, Theorem \ref{thm:subspace_decompose1} and Proposition \ref{prop:subspace_decompose2} tell us that for those factorization functions which do not have enough structure to belong to $\mathcal{U}_n^h$ their variance may be computed according to the most naive heuristic of randomness. Indeed, one last reformulation of Theorem \ref{thm:subspace_decompose1} may be seen to be the following:
for $v_1, v_2 \in \mathcal{V}_n^n$,
\begin{equation}
	\mathbb{E}_{\substack{\chi\, (T^{n-h}) \\ \text{prim., ev.}}}\; \sum_{\substack{f, g \in \mathcal{M}_n \\ f \neq g}} v_1(f) \chi(f) \overline{v_2(f) \chi(f)} = o(q^n).
\end{equation}
It would be interesting to see whether a modification of this picture is consistent with conjectures that have been made in other settings (e.g. in the fixed $q$ large $n$ limit, or over number fields), or indeed with statistics in orthogonal and symplectic families.

\section{Acknowledgments}

For helpful discussions, I thank a number of people, including Efrat Bank, Dan Bump, Reda Chhaibi, Paul-Olivier Dehaye, Andrew Granville, Jeff Lagarias, Zeev Rudnick, Will Sawin, and especially Ofir Gorodetsky, who made a number of careful suggestions and corrections to an earlier draft that have greatly improved the paper. I also want to thank Jordan Ellenberg, Daniel Hast, and Vlad Matei; the decomposition of Proposition \ref{prop:subspace_decomposedivisors} came out of discussions with them during a visit to Madison. A discussion on the website MathOverflow, available at \url{http://mathoverflow.net/q/233167} was useful for finding a reference. Part of this work was done while I was a postdoctoral fellow at the University of Z\"urich, and I thank that institution for its hospitality.

\end{document}